\documentclass[12pt,reqno]{amsart}

\usepackage{amsmath,amsthm,amssymb}

\usepackage{tikz-cd} 

\usepackage{graphicx, psfrag, amsmath, amscd, amssymb,mathtools}
\usepackage{lipsum}
\usepackage{subcaption} 

\usepackage[bookmarksnumbered, colorlinks, plainpages]{hyperref}
\hypersetup{colorlinks=true,linkcolor=red, anchorcolor=green, citecolor=cyan, urlcolor=red, filecolor=magenta, pdftoolbar=true}

\usepackage{shuffle} 

\usepackage{titletoc}

\usepackage{ytableau}

\usepackage[pagewise,displaymath, mathlines]{lineno}  

\newcommand{\C}{{\mathcal{C}}}

\newcommand{\G}{{\mathcal{G}}}

\newcommand{\Code}{{\mathtt{Code}}}

\newcommand{\DES}{{\mathsf{DES}}}
\newcommand{\des}{{\mathsf{des}}}

\newcommand{\inv}{{\mathsf{inv}}}
\newcommand{\exc}{{\mathsf{exc}}}
\newcommand{\EXC}{{\mathsf{EXC}}}
\newcommand{\fix}{{\mathsf{fix}}}
\newcommand{\fixTwo}{{\mathsf{fix_2}}}
\newcommand{\maj}{{\mathsf{maj}}}

\newcommand{\DEX}{{\mathsf{DEX}}}

\newcommand{\ch}{{\mathrm{ch}}}
\newcommand{\ps}{{\mathrm{ps}_q}}

\newcommand{\rk}{{\mathrm{rk}}}

\newcommand{\Alt}{{\mathsf{Alt}}}
\newcommand{\RAlt}{{\mathsf{RAlt}}}
\newcommand{\Aut}{{\mathrm{Aut}}}

\newcommand{\Hom}{{\mathrm{Hom}}}
\newcommand{\Even}{{\mathsf{Even}}}
\newcommand{\Odd}{{\mathsf{Odd}}}

\renewcommand{\S}{{\mathfrak{S}}}

\newcommand{\x}{{\mathbf{x}}}
\renewcommand{\L}{{\mathcal{L}}}

\newcommand{\Hilb}{{\mathsf{Hilb}}}
\newcommand{\grFrob}{{\mathsf{grFrob}}}
\newcommand{\qbin}[2]{{#1 \brack #2}_q} 

\newtheorem{thm}{Theorem}[section]
\newtheorem{cor}[thm]{Corollary}

\newtheorem{prop}[thm]{Proposition}
\newtheorem{lem}[thm]{Lemma}

\theoremstyle{definition}
\newtheorem{defi}[thm]{Definition}

\newtheorem{exa}[thm]{Example}

\theoremstyle{remark}
\newtheorem{rem}{Remark}[section]

\newtheoremstyle{TheoremNum}
   {\topsep}{\topsep}              
   {\itshape}                      
    {}                              
    {\bfseries}                     
    {.}                             
    { }                             
    {\thmname{#1}\thmnote{ \bfseries #3}}
\theoremstyle{TheoremNum}
\newtheorem{thmn}{Theorem}
\newtheorem{corn}{Corollary}

\title[(augmented) Chow rings of uniform matroids and their $q$-analogs]{Chow rings and augmented Chow rings of uniform matroids and their $q$-analogs}

\author{Hsin-Chieh Liao}
\address{Department of Mathematics, University of Miami, USA}
\email{h.liao@math.miami.edu}

\date{}

\begin{document}

\maketitle

\begin{abstract}
We study the Hilbert series and the representations of $\S_n$ and $GL_n(\mathbb{F}_q)$ on the (augmented) Chow rings of uniform matroids $U_{r,n}$ and $q$-uniform matroids $U_{r,n}(q)$. The Frobenius series for uniform matroids and their $q$-analogs are computed. As a byproduct, we recover Hameister, Rao, and Simpson's formula for the Hilbert series of Chow rings of $q$-uniform matroids in terms of permutations and further obtain their augmented counterpart in terms of decorated permutations.

We also show that the equivariant Charney--Davis quantity of the (augmented) Chow ring of a matroid is nonnegative (i.e., a genuine representation of a group of automorphisms of the matroid). When the matroid is a uniform matroid and the group is $\S_n$, the representation either vanishes or is a Foulkes representation (i.e., a Specht module of a ribbon shape). Specializing to the usual Charney--Davis quantities, we obtain an elegant combinatorial interpretation of Hameister, Rao, and Simpson's formula for Chow rings of $q$-uniform matroids and its augmented counterpart.
\end{abstract}

\dottedcontents{section}[0em]{}{0em}{0.5pc}
\dottedcontents{subsection}[0.5em]{}{0em}{0.5pc}

\setcounter{tocdepth}{1} 
\tableofcontents

\section{Introduction}
Matroids are combinatorial structures that unify various notions of independence across mathematics. Given their generality, they arise in a wide range of settings. Perhaps the most familiar example is from linear algebra, where a matroid abstracts the concept of linear independence of a collection of vectors; such matroids are called linear or realizable. Other important classes include algebraic matroids, coming from algebraic independence relations in field extensions, graphical matroids, whose bases are the forests of a graph, and transversal matroids, defined by partial matchings in a bipartite graph, etc.

Recently, the interactions between matroid theory and algebraic geometry have led to the resolution of several long-standing conjectures about unimodality, log-concavity, and top-heaviness of combinatorial sequences coming from matroids. 
In particular, Adiprasito, Huh, and Katz \cite{AHK2018} proved the long-standing Heron-Rota-Welsh conjecture, which asserts the log-concavity of the Whitney numbers of the first kind for a matroid.
Following the seminal paper \cite{AHK2018}, Braden, Huh, Matherne, Proudfoot, and Wang \cite{BHMPW2020+} resolved the Dowling--Wilson conjecture, which states that the Whitney numbers of the second kind of a matroid are top-heavy. 
The Chow ring $A(M)$ and the augmented Chow ring $\widetilde{A}(M)$ of a matroid $M$ play important roles in the respective proofs of the two conjectures. 
They were shown to satisfy the so-called ``\emph{K\"{a}hler package}'', which includes \emph{Poincar\'{e} duality}, the \emph{Hard Lefschetz property}, and the \emph{Hodge-Riemann relations}. 
Since then, extensive studies of the Chow rings of matroids and their extensions have been conducted, for example \cite{BHMPW2020,ADH2023,HameisterRaoSimpson2021,Mastroeni2022Koszul,Ferroni2022hilbert,EHL2022stellahedral,ANR2023PermRepchow}.

Among these developments, Hameister, Rao, and Simpson \cite{HameisterRaoSimpson2021} studied the Hilbert series of Chow rings of uniform matroids $U_{r,n}$ and their $q$-analogs $U_{r,n}(q)$, and their $(-1)$-evaluations, or equivalently, the \emph{Charney--Davis quantities} (see Section \ref{subsec:CD} for definitions). 
They related the Hilbert series of the Chow ring of the $q$-uniform matroid $U_{r,n}(q)$ to the enumeration of the set $\S_n$ of permutations on $[n]\coloneqq\{1,2,\ldots,n\}$. 

\begin{thm}\label{HRS:ChowQUniform} \cite[Theorem 1.1]{HameisterRaoSimpson2021}
For any positive integer $n$ and $1\le r\le n$,
\[
    \Hilb(A(U_{r,n}(q)),t)=\sum_{\sigma\in\S_n}q^{\maj(\sigma)-\exc(\sigma)}t^{\exc(\sigma)}-\sum_{j=r}^{n-1}\sum_{\substack{\sigma\in\S_n \\ \fix(\sigma)\ge n-j}}q^{\maj(\sigma)-\exc(\sigma)}t^{j-\exc(\sigma)},
\]  
where $\maj(\sigma)$ and $\exc(\sigma)$ are the classical major index and excedence number, and $\fix(\sigma)$ is the number of fixed points of a permutation $\sigma$.
\end{thm}
 In particular, when the rank $r$ is $n$ and $n-1$, the corresponding Hilbert series are given by the $q$-analogs of Eulerian polynomials $A_n(q,t)$ for $\S_n$ and $d_n(q,t)$ for the set $D_n$ of derangements studied by Shareshian and Wachs \cite{ShareshianWachs2010}. See Section \ref{subsubsec:Eulerian}, the \emph{Eulerian story}, for a detailed discussion of these polynomials.
 
\begin{cor}\label{HRS:qHilbChowSpecial} \cite[Corollary 3.3, 3.12]{HameisterRaoSimpson2021}
For the special case that $r=n$ and $r=n-1$ we have
\begin{itemize}
    \item[(i)] \(\displaystyle \Hilb(A(U_{n,n}(q)),t) =\sum_{\sigma\in\S_n}q^{\maj(\sigma)-\exc(\sigma)}t^{\exc(\sigma)}=A_n(q,t).\)
    \item[(ii)]\(\displaystyle \Hilb(A(U_{n-1,n}(q)),t) =\sum_{\sigma\in D_n}q^{\maj(\sigma)-\exc(\sigma)}t^{\exc(\sigma)-1}=t^{-1}d_n(q,t).\)
\end{itemize}
\end{cor}

For a positive integer $n$, define  $[n]_q\coloneqq 1+q+\ldots+q^{n-1}$, $[n]_q!\coloneqq\prod_{i=1}^n [i]_q$, and the \emph{$q$-binomial coefficient} $\qbin{n}{k}\coloneqq\frac{[n]_q!}{[k]_q![n-k]_q!}$.
Hameister, Rao, and Simpson also showed that the Charney--Davis quantity of $A(U_{r,n}(q))$ is given by
\begin{equation} \label{HRS:sign}
    \Hilb(A(U_{r,n}(q)),-1)=\begin{cases}
    0 & \mbox{ if $r$ is even}\\
    \displaystyle{\sum_{k=0}^{\frac{r-1}{2}}(-1)^k\qbin{n}{2k}E_{2k}(q)} & \mbox{ if $r$ is odd}
    \end{cases},
\end{equation}
where $E_{2k}(q)=\sum_{\sigma\in\RAlt_{2k}}q^{\inv(\sigma)}$ is a $q$-analog of the $k$-th \emph{secant number} and $\RAlt_{2k}=\{\sigma\in\S_{2k}\,:\, \sigma_1<\sigma_2>\ldots>\sigma_{2k-1}<\sigma_{2k}\}$ the set of \emph{reverse alternating permutations}.\footnote{In this paper, we follow the terminology in Stanley's survey \cite{StanleyAlternatingSurvey} calling such permutations \emph{reverse alternating} and the permutations with $\sigma_1>\sigma_2<\sigma_3\ldots$ \emph{alternating}. 
In \cite[Remark 4.5]{HameisterRaoSimpson2021}, they reverse the terminology and the $q^{\exc(\sigma)}$ in the expression of $E_{2k}(q)$ is a typo.}
The alternating sum in the case where $r$ is odd in (\ref{HRS:sign}) can also be expressed as an alternating sum of determinants:
\begin{equation} \label{HRS:determinant}
1+[n]_q!\sum_{k=0}^{\frac{r-1}{2}}\frac{(-1)^k}{[n-2k]_q!}
\left|\begin{array}{ccccc}
\frac{1}{[2]_q!} &         1        &   0   & \ldots & 0 \\
\frac{1}{[4]_q!} & \frac{1}{[2]_q!} &   1   & \ldots & 0 \\
\vdots           & \vdots           & \vdots& \ddots & \vdots\\
\frac{1}{[2k-2]_q!} & \frac{1}{[2k-4]_q!} & \frac{1}{[2k-6]_q!} & \ldots & 1\\
\frac{1}{[2k]_q!} & \frac{1}{[2k-2]_q!}& \frac{1}{[2k-4]_q!} & \ldots & \frac{1}{[2]_q!}
\end{array}\right|.
\end{equation}

In this paper, we continue the study of Chow rings of uniform matroids and $q$-uniform matroids initiated by Hameister, Rao, and Simpson. 
Using the Feichtner--Yuzvinsky type basis for the augmented Chow ring $\widetilde{A}(M)$ found by the author \cite[Corollary 5.11]{Liao2023One}, and independently by Eur, Huh, and Larson \cite[Lemma 7.8]{EHL2022stellahedral}, we build the parallel story for the augmented Chow rings of $U_{r,n}$ and $U_{r,n}(q)$.
Besides their Hilbert series and corresponding Charney--Davis quantities, we further explore their equivariant extensions. Our main results can be divided into three parts: Hilbert series for $q$-uniform matroids, Frobenius series for uniform matroids, and equivariant Charney--Davis quantities for general matroids.

\subsection{Hilbert series: $q$-uniform matroids} \label{subsec:q-uniform} 
The augmented counterpart we establish involves the combinatorics parallel to the Eulerian story, which we refer to as the \emph{binomial Eulerian story} introduced in Section \ref{subsubsec:BinomialEulerian}. For example, this includes decorated permutations $\widetilde{\S}_n$, binomial Eulerian polynomials $\widetilde{A}_n(t)$ and their $q$-analogs $\widetilde{A}_n(q,t)$, binomial Eulerian quasisymmetric function $\widetilde{Q}_n(\x,t)$, etc.

\begin{thm} \label{thm:HilbAugQUniform}
 The Hilbert series $\Hilb(\widetilde{A}(U_{r,n}(q)),t)$ has the following two expressions
 \begin{itemize}
     \item[(i)] \(\displaystyle 1+t\sum_{j=0}^{r-1}\qbin{n}{j}A_j(q,t)(1+t+\ldots+t^{r-j-1})\),
     \item[(ii)] \(\displaystyle\sum_{\sigma\in\widetilde{\S}_n}q^{\maj(\sigma)-\exc(\sigma)}t^{\exc(\sigma)+1}-\sum_{j=r}^{n-1}\sum_{\substack{
    \sigma\in\widetilde{\mathfrak{S}}_n\\
    \fixTwo(\sigma)\ge n-j}}q^{\maj(\sigma)-\exc(\sigma)}t^{j-\exc(\sigma)}\).
 \end{itemize}
\end{thm}

The first formula (Theorem \ref{thm:HilbAugQUniform}(i)) is a $q$-analog of the formula for the augmented Chow ring of arbitrary uniform matroids in recent work by Ferroni, Matherne, Stevens, and Vecchi in \cite[Second formula in Theorem 1.9]{Ferroni2022hilbert}. 

The second formula (Theorem \ref{thm:HilbAugQUniform}(ii)) bears much resemblance to Hameister, Rao, and Simpson's formula in Theorem \ref{HRS:ChowQUniform}. In particular, when the rank $r$ is $n$ and $n-1$, the following formulas are analogous to those in Corollary \ref{HRS:qHilbChowSpecial}.

\begin{cor} \label{cor:SpeicalCaseHilbAug}
For the special case that $r=n$ and $r=n-1$ we have
\begin{itemize}
    \item[(i)]\(\displaystyle \Hilb(\widetilde{A}(U_{n,n}(q)),t)=\sum_{\sigma\in\widetilde{\S}_n}q^{\maj(\sigma)-\exc(\sigma)}t^{\exc(\sigma)+1}=\widetilde{A}_n(q,t)\),
    \item[(ii)]\(\displaystyle\Hilb(\widetilde{A}(U_{n-1,n}(q)),t)=\sum_{\sigma\in\S_n}q^{\maj(\sigma)-\exc(\sigma)}t^{\exc(\sigma)}=A_n(q,t)\).
\end{itemize}    
\end{cor}

For the non-augmented case, we also find a different expression of Theorem \ref{HRS:ChowQUniform} similar to Theorem \ref{thm:HilbAugQUniform} (i). It is a $q$-analog of Ferroni, Matherne, Stevens, and Vecchi's formula for the Chow ring of arbitrary uniform matroids in \cite[first formula in Theorem 1.9]{Ferroni2022hilbert}. 

\begin{thm} \label{thm:HilbQUniform}
For any positive integer $n$ and $1\le r\le n$,
\[
    \Hilb(A(U_{r,n}(q)),t)=\sum_{j=0}^{r-1}\qbin{n}{j}d_j(q,t)(1+t+\ldots+t^{r-j-1}).
\]
\end{thm}

We note that all our results hold as stated in the case $q=1$ if we view $U_{r,n}(1)\coloneqq U_{r,n}$.

\subsection{Frobenius series: uniform matroids} \label{subsec:uniform}
Parallel to the $q$-uniform matroids, we study the representation of $\S_n$ on the Chow rings and the augmented Chow rings of uniform matroids. 
These representations are encoded as their \emph{graded Frobenius series}, which are polynomials with symmetric function coefficients (See Section \ref{subsec:SymRep}).
For this purpose, we pass the underlying fields of the Chow rings to $\mathbb{C}$. For a matroid $M$, we write $A(M)_{\mathbb{C}}\coloneqq\mathbb{C}\otimes_{\mathbb{Q}}A(M)$ and  $\widetilde{A}(M)_{\mathbb{C}}\coloneqq\mathbb{C}\otimes_{\mathbb{Q}}\widetilde{A}(M)$. The symmetric function formulas that we obtained involve variations of \emph{Eulerian quasisymmetric functions} $Q_n^0(\x,t)$, $Q_n(\x,t)$, $\widetilde{Q}_n(\x,t)$, which can be viewed as $\S_n$-equivariant enumerators of derangements $D_n$, permutations $\S_n$, and decorated permutations $\widetilde{\S}_n$, respectively. Their explicit expressions are in Definition \ref{def:EulerianQuasi} and Theorem \ref{thm:FexpWideQ}. See Table \ref{table:relations} in Section \ref{subsubsec:BinomialEulerian} for the relationships between these symmetric functions and variations of Eulerian polynomials.

\begin{thm} \label{thm:FrobChowUniform}
The graded Frobenius series of the Chow ring of an arbitrary uniform matroid $\grFrob(A(U_{r,n})_{\mathbb{C}},t)$ has two expressions:
\begin{itemize}
    \item[(i)] \(\displaystyle \sum_{j=0}^{r-1}h_{n-j}Q_j^0(\x,t)(1+t+\cdots+t^{r-j-1})\),
    \item[(ii)] \(\displaystyle Q_n(\x,t)-\sum_{j=r}^{n-1}\sum_{\substack{\sigma\in\S_n\\ \fix(\sigma)\ge n-j}}F_{\DEX(\sigma),n}t^{j-\exc(\sigma)}\).
\end{itemize}
\end{thm}

\begin{cor} \label{cor:FrobChowSpecial}
For the special case that $r=n$ and $r=n-1$, we have
\begin{itemize}
    \item[(i)] \(\displaystyle \grFrob(A(U_{n,n})_{\mathbb{C}},t)=Q_n(\x,t),\)
    \item[(ii)]\(\displaystyle \grFrob(A(U_{n-1,n})_{\mathbb{C}},t)=t^{-1}Q_n^0(\x,t).\)
\end{itemize}
\end{cor}

\begin{thm} \label{thm:FrobAugUniform}
The graded Frobenius series of the augmented Chow ring of an arbitrary uniform matroid $\grFrob(\widetilde{A}(U_{r,n})_{\mathbb{C}},t)$ has two expressions:
\begin{itemize}
    \item[(i)] \(\displaystyle h_n+t\sum_{j=0}^{r-1}h_{n-j}Q_j(\x,t)(1+t+\cdots+t^{r-j-1})\),
    \item[(ii)] \(\displaystyle \widetilde{Q}_n(\x,t)-\sum_{j=r}^{n-1}\sum_{\substack{\sigma\in\widetilde{\S}_n\\ \fixTwo(\sigma)\ge n-j}}F_{\DEX(\sigma),n}t^{j-\exc(\sigma)}\),
\end{itemize}
where $\fixTwo(\sigma)$ is the number of $0$s in the one-line notation of the decorated permutation $\sigma$.
\end{thm}

\begin{cor} \label{cor:FrobAugSpecial}
For the special case that $r=n$ and $r=n-1$, we have
\begin{itemize}
    \item[(i)] \(\displaystyle \grFrob(\widetilde{A}(U_{n,n})_{\mathbb{C}},t)=\widetilde{Q}_n(\x,t).\)
    \item[(ii)]\(\displaystyle \grFrob(\widetilde{A}(U_{n-1,n})_{\mathbb{C}},t)=Q_n(\x,t).\)
\end{itemize}
\end{cor}

\begin{rem} \label{rem:BooleanIsProved}
Corollary \ref{cor:FrobChowSpecial} (i) (resp. \ref{cor:FrobAugSpecial} (i)) was proved in \cite[Theorem 6.9]{Liao2023One} (resp. in \cite[Theorem 6.22]{Liao2023One}) by constructing an $\S_n$-equivariant bijection between Feichtner--Yuzvinsky basis $FY(U_{n,n})$ and the Stembridge codes introduced by Stembridge \cite{Stembridge1992} (resp. between $\widetilde{FY}(U_{n,n})$ and the extended codes introduced in \cite{Liao2023One}).
\end{rem}


Combining Theorem \ref{thm:FrobAugUniform} (i) and Corollary \ref{cor:FrobAugSpecial} (ii) in the case of $r=n-1$ gives the following recurrence relation for the Eulerian quasisymmetric function $Q_n(\x,t)$
\[
    Q_n(\x,t)=h_n+t\sum_{j=0}^{n-2}h_{n-j}Q_j(\x,t)(1+t+\cdots+t^{n-j-2}),
\]
which appears in Shareshain and Wachs' work in \cite[Corollary 4.1 when $r=1$]{ShareshianWachs2010}.

While it is not obvious, the $q$-polynomials for $q$-uniform matroids in Theorems \ref{HRS:ChowQUniform}, \ref{thm:HilbQUniform}, \ref{thm:HilbAugQUniform} and Corollaries \ref{HRS:qHilbChowSpecial}, \ref{cor:SpeicalCaseHilbAug} can be obtained from the symmetric function formulas for uniform matroids in Theorems \ref{thm:FrobChowUniform}, \ref{thm:FrobAugUniform} and Corollaries \ref{cor:FrobChowSpecial}, \ref{cor:FrobAugSpecial} respectively, by applying the stable principal specialization operator $\prod_{i=1}^n(1-q^i)\ps(-)$ (see Section \ref{subsubsec:Eulerian}). This phenomenon is stated in Theorem \ref{thm:psToDim} and is proved in Section \ref{Sec:Parellel} via Steinberg's theory of unipotent representations of $GL_n(\mathbb{F}_q)$.

\subsection{Equivariant Charney--Davis quantities for general matroids} \label{subsec:EQCD}
In our study of the Charney--Davis quantities, we generalize Hameister, Rao, and Simpson's formula (\ref{HRS:sign}) to arbitrary matroids in an equivariant setting. A unified interpretation of the Charney--Davis quantities of the Chow ring and the augmented Chow ring of an arbitrary matroid is obtained. See Section \ref{subsec:CD} for more details on the Charney--Davis quantity and its equivariant extension, and Section \ref{subsec:RankSelection} for the poset homology of a rank-selected poset.

Denote by $\Even(k)$ the set of even numbers from $2$ up to $k$ for an even integer $k$, and by $\Odd(\ell)$ the set of odd numbers from $1$ up to $\ell$ for an odd integer $\ell$, i.e.
\[
    \Even(k)=\{2,4,6,\ldots,k\},\quad \Odd(\ell)=\{1,3,5,\ldots,\ell\}.
\]
Let $G$ be a subgroup of the automorphism group $\Aut(M)$ of a loopless matroid $M$. 
Consider the $G$-equivariant Hilbert series of $A(M)_{\mathbb{C}}$. 
Then the following theorem gives an interpretation for the $G$-equivariant Charney--Davis quantity for the Chow ring of a matroid.

\begin{thm} \label{thm:EqCDChow}
Let $M$ be a (loopless) matroid on $[n]$ of rank $r$, then
\[
    \sum_{i=0}^{r-1}(-1)^i A^i(M)_{\mathbb{C}}\cong_{G}
    \begin{cases}
        0   &   \text{ if $r$ is even}\\
        (-1)^{\frac{r-1}{2}}\tilde{H}_{\frac{r-3}{2}}\left(\L(M)_{\Even(r-1)}\right)
         & \text{ if $r$ is odd}
    \end{cases}.
\]
In particular, the $G$-equivariant Charney--Davis quantity $CD^G(A(M))$ is nonnegative in $R_{\mathbb{C}}(G)$, i.e., it is a genuine representations of $G$.
\end{thm}

For the augmented case, an interpretation for the $G$-equivariant Charney--Davis quantity is given by the following theorem.

\begin{thm} \label{thm:EqCDAug}
Let $M$ be a matroid on $[n]$ of rank $r$. Then
\[
    \sum_{i=0}^{r}(-1)^i\widetilde{A}^i(M)_{\mathbb{C}}\cong_{G}
    \begin{cases}
        (-1)^{\frac{r}{2}}\tilde{H}_{\frac{r-2}{2}}\left(\L(M)_{\Odd(r-1)}\right)   &   \text{ if $r$ is even}\\
        0   & \text{ if $r$ is odd}
    \end{cases}.
\]
In particular, the $G$-equivariant Charney--Davis quantity $CD^G\left(\widetilde{A}(M)\right)$ is nonnegative in $R_{\mathbb{C}}(G)$.
\end{thm}

The cases when $M$ is a uniform matroid $U_{r,n}$ and a $q$-uniform matroid $U_{r,n}(q)$ are treated immediately after the proofs of Theorems \ref{thm:EqCDChow} and \ref{thm:EqCDAug} in Section \ref{subsec:CDUniform}. 
In particular, we show that (\ref{HRS:sign}) admits an elegant combinatorial interpretation, and that its augmented analog likewise admits a compatible combinatorial interpretation (See Theorem \ref{thm:SignedQUniform}).

The paper is organized as follows. 
Section \ref{sec:Background} introduces preliminary definitions and background. Sections \ref{Sec:ProofFrobChow} and \ref{Sec:ProofFrobAug} establish our Frobenius-series formulas for the Chow rings and the augmented Chow rings of uniform matroids, respectively. 
 In Section \ref{Sec:Parellel}, we describe the parallelism between $\S_n$-representations and unipotent $GL_n(\mathbb{F}_q)$-representations, and apply it to the (augmented) Chow rings of uniform matroids and $q$-uniform matroids. 
Consequently, the $q$-uniform case reduces to the uniform case handled in the preceding sections.
Section \ref{Sec:GeneralEQ} presents our general results on Charney–Davis quantities. In Section \ref{subsec:generalCD}, we prove nonnegativity of the equivariant Charney–Davis quantity for arbitrary matroids. Section \ref{subsec:CDUniform} examines the Charney–Davis quantity for uniform matroids in the classical, $q$-uniform, and equivariant settings. Finally, Section \ref{sec:Final} outlines ongoing work extending these results to $\gamma$-positivity.

\section{Background} \label{sec:Background}

\subsection{Matroids} \label{subsec:Matroids}
We first briefly review some basics of matroids. 
For a more thorough introduction to matroids, we refer the readers to Oxley \cite{Oxley2006matroid}. Some familiarity with terminologies on partially ordered sets (posets) is assumed. See Stanley \cite[Chapter 3]{StanleyEC1} for a good reference. 
There are many ways to define a matroid, and they are \emph{cryptomorphic} to each other (equivalent, but not obviously). Due to our purpose, here we define a matroid in terms of its \emph{flats}.  

\begin{defi}
A matroid $M=(E,\L(M))$ consists of a finite ground set $E$ and a collection of subsets $\L(M)\subseteq 2^{E}$ satisfy the following axioms:
\begin{itemize}
    \item[($F_1$)] $E\in \L(M)$.
    \item[($F_2$)] If $F$, $F'\in\L(M)$, then $F\cap F'\in\L(M)$.
    \item[($F_3$)] For any $F\in\L(M)$ and any $i\in E\setminus F$, there exists a unique $F'$ containing $i$ and \emph{covering} $F$, in the sense that $F\subsetneq F'$ and no other $F''\in \L(M)$ such that $F\subsetneq F''\subsetneq F'$.
\end{itemize}
\end{defi}

The subsets in $\L(M)$ are called the \emph{flats} of $M$. 
These axioms are the abstraction of the model where $E$ is a collection of vectors $v_1,\ldots,v_n$ in a vector space over some field $\mathbb{F}$, and $\L(M)$ is the set of subspaces $\{\mathrm{span}_{\mathbb{F}}\{v_i\}_{i\in S}: S\subseteq E\}$ in which each subspace $W$ is represented by the maximal subcollection of $v_1,\ldots,v_n$ that spans $W$. 

The axioms $(F_1)$, $(F_2)$ imply that $\L(M)$ with the inclusion order forms a lattice with \emph{meet} $F\wedge F'=F\cap F'$ and \emph{join} $F\vee F'=\bigcap_{F''\supseteq F,F'}F''$ for any $F$, $F'\in\L(M)$. 
The axiom $(F_3)$ further implies that $\L(M)$ is a \emph{geometric lattice}. We call $\L(M)$ the \emph{lattice of flats of $M$}. It is well-known \cite[Chapter 3]{StanleyEC1} that any finite geometric lattice is a lattice of flats of some matroid.

Denote $\rk_M$ the rank function of $\L(M)$. We say the \emph{rank} of the matroid $M$ is $r$ if $\rk_M(E)=r$. 
For any subset $S\subseteq E$, the closure $\overline{S}$ of $S$ is the minimal flat containing $S$; in other words, $\overline{S}=\bigcap_{F\supseteq S, F\in\L(M)}F$. Then one can extend the rank function $\rk_M$ to $2^{E}\rightarrow \mathbb{N}$ by defining $\rk_M(S)\coloneqq\rk_M(\overline{S})$ for any $S\subseteq E$. We call an element $i\in E$ a \emph{loop} if $\rk_M(i)=0$.

A permutation $\sigma$ on $E$ is an \emph{automorphism of the matroid} $M$ if $\rk_M(\sigma(X))=\rk_M(X)$ for all subsets $X\subseteq E$. 
Denote $\Aut(M)$ the set of automorphisms of $M$. 
In particular, an automorphism $\sigma\in\Aut(M)$ maps a flat $F\in\L(M)$ to a flat  $\sigma(F)\in\L(M)$ and $\rk_M(\sigma(F))=\rk_M(F)$ for all $F\in\L(M)$. Hence, an automorphism $\sigma$ permutes the chains in $\L(M)$.

In this paper, most of the time we will focus on the uniform matroids and their $q$-analogs. 
The \emph{uniform matroid} $U_{r,n}$ ($r\le n$) is a matroid on $E=[n]$ of rank $r$ whose collection of flats are  
\[
    \{S\subseteq [n]: |S|\le r-1\}\cup\{[n]\}.
\]
The lattice of flats $\L(U_{r,n})$ is a truncated Boolean lattice. 
In particular, when $r=n$, the lattice $\L(U_{n,n})$ is the Boolean lattice $B_n=(2^{[n]},\subseteq)$, hence $U_{n,n}$ is also called a Boolean matroid.

\begin{exa} \label{exa:uniform2}
The lattices of flats of uniform matroids $U_{r,n}$ when $n=3$ are the truncated Boolean lattices as shown in Figure \ref{fig:uniform}.

\begin{figure}[h]
    \begin{subfigure}{0.3\textwidth}
        \includegraphics{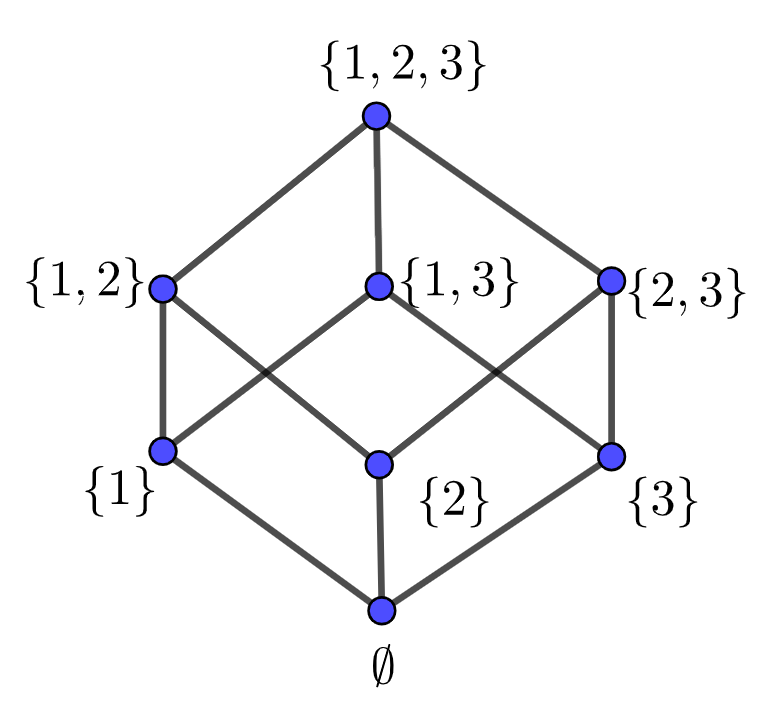}
        \caption{Boolean lattice $\L(U_{3,3})$}
        \label{fig:B_33}    
    \end{subfigure}
    \hspace{1 cm}
    \begin{subfigure}{0.3\textwidth}
        \includegraphics{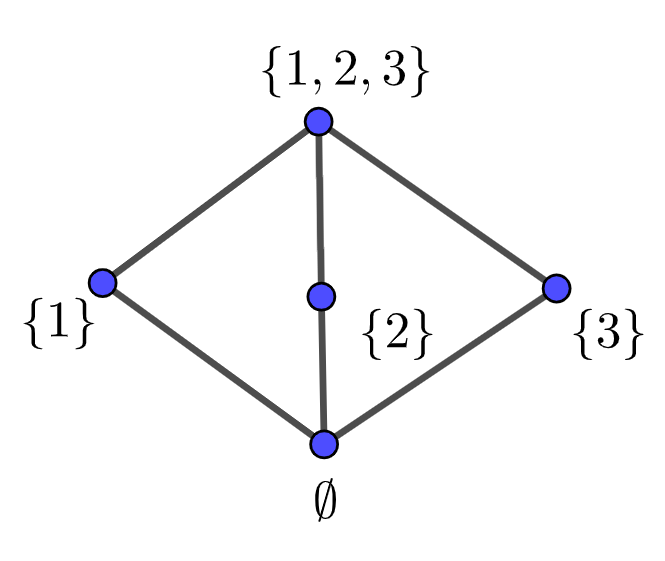}
        \caption{$\L(U_{2,3})$}
        \label{fig:B_32}        
    \end{subfigure}
    \hspace{.5 cm}
    \begin{subfigure}{0.2\textwidth}
        \includegraphics{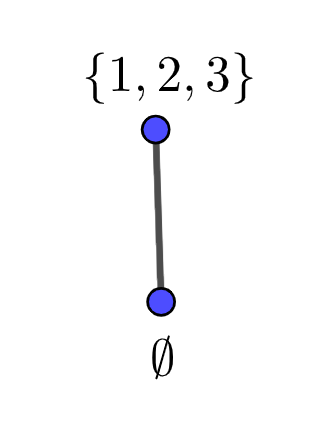}
        \caption{$\L(U_{1,3})$}
        \label{fig:B_31}          
    \end{subfigure}
\caption{}
\label{fig:uniform}
\end{figure}
\end{exa}

Let $q>1$ be a prime power. 
The \emph{$q$-uniform matroid}, which we refer to as $U_{r,n}(q)$, is a natural $q-$analog of $U_{r,n}$. 
It is a matroid whose ground set is the vector space $\mathbb{F}_q^n$ over the finite field $\mathbb{F}_q$. 
Its lattice of flats $\L(U_{r,n}(q))$ consists of subspaces $V\le\mathbb{F}_q^n$ of $\dim V\le r-1$ and the maximal subspace $\mathbb{F}_q^n$. 
It will be convenient to define $U_{r,n}(1)\coloneqq U_{r,n}$.
In particular, the lattice of flats of $U_{n,n}(q)$ is the subspace lattice of $\mathbb{F}_q^n$, also known as the $q$-Boolean lattice $B_n(q)$, which is a natural $q$-analog of the Boolean lattice $B_n$. 

\subsection{Chow rings and augmented Chow rings of matroids} \label{subsec:ChowRings}

In the following, we assume $M$ to be a loopless matroid.
\begin{defi}
Let $S$ be the polynomial ring $\mathbb{Q}[x_F:F\in\L(M)\setminus\{\emptyset\}]$. The \emph{Chow ring} $A(M)$ of a matroid $M$ is a quotient $\mathbb{Q}$-algebra
\[
    A(M)\coloneqq S/(I+J)
\]
where $I$ and $J$ are the following ideals of $S$:
\[
    I=\left\langle x_F x_{F'}: F, F'\text{ are incomparable in }\L(M)\right\rangle \quad \text{and} \quad     J=\left\langle \sum_{F:F\ni i}x_F: i\in E\right\rangle.
\]
\end{defi}
This Chow ring was introduced and studied by Feichtner and Yuzvinsky \cite{FeichtnerYuzvinsky2004}, who defined a Chow ring for any atomic lattice $\L$ equipped with a ``building set" $\G\subset\L$. (The Chow ring $A(M)$ is the Chow ring of the lattice $\L(M)$ of flats with the maximal building set.) In particular, they found a Gr\"{o}bner basis for the ideal $I+J$, thereby obtaining a basis for $A(M)$ consisting of the standard monomials, which we refer to as the \emph{Feichtner--Yuzvinsky basis} (FY-basis).

\begin{thm} \cite[Corollary 1]{FeichtnerYuzvinsky2004} \label{thm:FYChowBasis}
The following set of monomials forms a basis for $A(M)$
\begin{equation} \label{eq:FYChowBasis}
    FY(M)\coloneqq\left\{x_{F_1}^{a_1}x_{F_2}^{a_2}\ldots x_{F_\ell}^{a_\ell}:\substack{~\emptyset=F_0\subsetneq F_1\subsetneq F_2\subsetneq\ldots\subsetneq F_\ell \text{ for any } \ell \\
    1\le a_i\le \rk_M(F_i)-\rk_M(F_{i-1})-1}\right\}.
\end{equation}    
\end{thm}
We write $FY^j(M)$ as the subset of the degree $j$ monomials in $FY(M)$. 

To prove the Dowling--Wilson conjecture, Braden, Huh, Matherne, Proudfoot, and Wang \cite{BHMPW2020,BHMPW2020+} introduced an extension of the Chow ring of a matroid called the \emph{augmented Chow ring}.
\begin{defi} \cite{BHMPW2020, BHMPW2020+}
Let $\widetilde{S}$ be the polynomial ring $\mathbb{Q}\left[\{x_F\}_{F\in\L(M) \setminus\{E\}}\cup\{y_i\}_{i\in E}\right]$. The \emph{augmented Chow ring} of $M$ is 
\[
    \widetilde{A}(M)\coloneqq\widetilde{S}/(\widetilde{I}_1+\widetilde{I}_2+\widetilde{J}).
\]
where $\widetilde{I}_1$, $\widetilde{I}_2$, and $\widetilde{J}$ are the following ideals of $\widetilde{S}$:
\begin{itemize}
    \item $\widetilde{I}_1$ is generated by $x_F x_{F'}$ for any proper flats $F$ and $F'$.
    \item $\widetilde{I}_2$ is generated by $y_i x_F$ for any $i\in E$ and any proper flats $F$ such that $i\notin F$.
    \item $\widetilde{J}$ is generated by the linear elements $y_i-\sum_{F:i\notin F}x_F$ for all $i\in E$.
\end{itemize}
\end{defi}

In the author's previous work \cite{Liao2023One}, and independently by Eur, Huh, and Larson in \cite{EHL2022stellahedral}, it was shown that the augmented Chow ring $\widetilde{A}(M)$ can be realized as the Chow ring of the lattice of flats of the free coextension $M\times e$ of $M$ with respect to a certain building set. By \cite[Corollary 1]{FeichtnerYuzvinsky2004} the augmented Chow ring has the following monomial basis. 

\begin{thm} \cite{Liao2023One,EHL2022stellahedral} \label{thm:FYAugBasis}
The following set of monomials forms a basis for $\widetilde{A}(M)$
\begin{equation} \label{eq:AugChowBasis}
	\widetilde{FY}(M)\coloneqq\left\{x_{F_1}^{a_1}x_{F_2}^{a_2}\ldots x_{F_{\ell}}^{a_{\ell}}: \substack{\emptyset\subsetneq F_1\subsetneq F_2\subsetneq\ldots\subsetneq F_{\ell} \text{ for any }\ell \\
    1\le a_1\le\rk_M(F_1),~ a_i\le\rk_M(F_i)-\rk_M(F_{i-1})-1 \text{ for }i\ge 2}
    \right\}.
\end{equation}
\end{thm}
Denote by $\widetilde{FY}^j(M)$ the subset of $\widetilde{FY}(M)$ consisting of the monomials of degree $j$.

When $M$ is the uniform matroid $U_{r,n}$, we can consider $\S_n$ acting on the polynomial ring $S$ and $\widetilde{S}$ by permuting the variables $x_F$ which are indexed by subsets $F$ of $[n]$ and $y_i$ for $i\in [n]$. 
Both of the ideals $I+J$ of $S$ and $\widetilde{I}_1+\widetilde{I}_2+\widetilde{J}$ of $\widetilde{S}$ are generated by homogeneous polynomials, and are invariant under the action of $\S_n$. 
Hence the Chow rings $A(U_{r,n})$ and $\widetilde{A}(U_{r,n})$ and each of their degree components form $\S_n$-modules. 
Furthermore, the bases $FY(U_{r,n})$ in (\ref{eq:FYChowBasis}) and $\widetilde{FY}(U_{r,n})$ in (\ref{eq:AugChowBasis}) are permutation bases for the $\S_n$-modules $A(U_{r,n})$ and $\widetilde{A}(U_{r,n})$, respectively. 
Similarly, when $M$ is the $q$-uniform matroid $U_{r,n}(q)$, the rings $A(U_{r,n}(q))$, $\widetilde{A}(U_{r,n}(q))$ are $GL_n(\mathbb{F}_q)$-modules with permutation bases $FY(U_{r,n}(q))$, $\widetilde{FY}(U_{r,n}(q))$ respectively.

\subsection{Charney--Davis quantity and its equivariant extension} \label{subsec:CD}
Let $\mathbb{F}$ be any field and $A=\bigoplus_{j\in\mathbb{N}}A^j$ be an $\mathbb{N}$-graded $\mathbb{F}$-algebra such that each degree-$j$ homogeneous component $A^j$ ($j\in\mathbb{N}$) satisfies $\dim_{\mathbb{F}}A^j<\infty$.  
Then the Hilbert series of $A$ is defined to be 
\[
    \Hilb(A,t)\coloneqq \sum_{j\in\mathbb{N}}\dim_{\mathbb{F}}(A^j)t^j.
\]

We say that the Hilbert series of $A$ is \emph{palindromic}, if there is some $n\ge 0$ such that $\dim_{\mathbb{F}}(A^N)=0$ for $N>n$ and $\Hilb(A,t)=\sum_{j=0}^n a_jt^j$ with $a_j=a_{n-j}$ for $j\le n/2$. When the Hilbert series of $A$ is a polynomial of degree $n$, we call the number
\[
    CD(A)\coloneqq
    \begin{cases}
        (-1)^{\frac{n}{2}}\Hilb(A,-1) & \text{ if $n$ is even}\\
        \Hilb(A,-1) & \text{ if $n$ is odd}
    \end{cases}
\]
the \emph{Charney--Davis quantity} of $A$. In particular, if $A$ has a palindromic Hilbert series of odd degree, then $CD(A)=0$.

This quantity was first introduced by Charney and Davis in \cite{CharneyDavis1995} in connection with the local form of the Hopf conjecture for piecewise Euclidean cubical manifolds. 
The conjecture has an equivalent form stating that the $h$-polynomial $h_{\Delta}(t)$ of a simplicial $(2n-1)$-sphere $\Delta$ satisfies $(-1)^nh_{\Delta}(-1)\ge 0$ if $\Delta$ is a \emph{flag} simplicial complex. 
The value $(-1)^n h_{\Delta}(-1)$ is now called the Charney--Davis quantity of $\Delta$, and the conjecture itself is known as the \emph{Charney--Davis conjecture}. 
For further details, see \cite{Zheng2019face} and \cite[Section 3]{AthanasiadisGamma}.

The Chow ring $A(M)$ and the augmented Chow ring $\widetilde{A}(M)$ admit alternative interpretations of their Charney--Davis quantities.  
For a matroid $M$ of rank $r$, the Hilbert series of $A(M)$ and $\widetilde{A}(M)$ are palindromic of degree $r-1$ and $r$ respectively (see Remark \ref{rem:palindromic}). Consequently:
\begin{itemize}
    \item If $r$ is even, then $CD(A(M))=0$; if $r$ is odd, then $CD(A(M))$ equals the signature of the quadratic form on $A^{\frac{r-1}{2}}(M)$ (see Remark \ref{rem:QuadForm}).
    \item Likewise, if $r$ is odd, then $CD(\widetilde{A}(M))=0$; if $r$ is even, then $CD(\widetilde{A}(M))$  equals the signature of the quadratic form on $\widetilde{A}^{\frac{r}{2}}(M)$.
\end{itemize}

\begin{rem} \label{rem:palindromic}
Adiprasito, Huh, and Katz \cite{AHK2018} and Braden, Huh, Matherne, Proudfoodt, and Wang \cite{BHMPW2020+} showed that $A(M)$ and $\widetilde{A}(M)$ satisfy the so-called \emph{K\"{a}hler package}, which includes Poincar\'{e} duality, Hard Lefschetz theorem, and Hodge-Riemann relations. 
In particular, Poincar\'{e} duality for $A(M)$ asserts that for each $0\le i\le r-1$ the multiplication map 
\[
    A^i(M)\times A^{r-1-i}(M)\longrightarrow A^{r-1}(M)
\]
induce an isomorphism 
\[
    A^{r-1-i}(M)\longrightarrow \Hom_{\mathbb{Q}}(A^i(M),A^{r-1}(M)).
\]
Composing this with a degree map $\deg: A^{r-1}(M)\longrightarrow\mathbb{Q}$ gives $A^{r-1-i}(M)\cong\Hom_{\mathbb{Q}}(A^i(M),\mathbb{Q})$, so $\dim_{\mathbb{Q}}(A^i(M))=\dim_{\mathbb{Q}}(A^{r-1-i}(M))$ for all $i$. 
Hence, the Hilbert series of $A(M)$ is palindromic, and the same argument applies to $\widetilde{A}(M)$.
\end{rem}

\begin{rem} \label{rem:QuadForm}
In the non-augmented case, if $r$ is odd, the multiplication map 
\[
    Q:A^{\frac{r-1}{2}}(M)\times A^{\frac{r-1}{2}}(M)\longrightarrow A^{r-1}(M)\cong\mathbb{Q}, \qquad Q(x)=x^2
\]
defines a quadratic form 
on $A^{\frac{r-1}{2}}(M)$ whose signature, by the the Hodge-Riemann relation for $A(M)$ (cf. \cite[proof of Theorem 1.1]{LeungReinerToric2002}), is equal to the Charney--Davis quantity of $A(M)$. 
Similarly, in the augmented case, when $r$ is even, the analogous form on $\widetilde{A}^{\frac{r}{2}}(M)$ has signature equal to the Charney–Davis quantity of $\widetilde{A}(M)$.
\end{rem}

Suppose now $\mathbb{F}$ is the field of complex numbers $\mathbb{C}$ and the graded algebra $A=\bigoplus_{i=0}^n A^i$ further carries a graded $\mathbb{C}G$-module structure for a finite group $G$. Let $R_{\mathbb{C}}(G)$ be the representation ring (or virtual character ring) of $G$ and the $G$-equivariant Hilbert series $\Hilb^G(A,t)=\sum_{i=0}^n A^it^i$ be a polynomial in $R_{\mathbb{C}}(G)[t]$. Then we may define the \emph{$G$-equivariant Charney--Davis quantity} of $A$ as 
\[
    CD^G(A)\coloneqq\begin{cases}
        (-1)^{\frac{n}{2}}\Hilb^G(A,-1) & \text{ if $n$ is even}\\
        \Hilb^G(A,-1)  & \text{ if $n$ is odd}
    \end{cases}.
\]
The resulting quanitiy $CD^G(A)$ lies in $R_{\mathbb{C}}(G)$. We say that $CD^G(A)$ is \emph{nonnegative} if it is from a representation of $G$. One can view $CD^G(A)$ as an alternating sum of the corresponding characters of $G$, and it is nonnegative if the result of the alternating sum indeed gives a character of a representation of $G$.

\subsection{Symmetric functions and representations of symmetric groups} \label{subsec:SymRep}

We assume some familiarity with symmetric functions and the representation theory of the symmetric group $\S_n$. 
Below, we briefly review how the \emph{ring of symmetric functions} $\Lambda$ encodes the representations of symmetric groups; see Sagan \cite{Sagan2013symmetric} and Stanley \cite{StanleyEC2} for references. 

Let $\Lambda_{\mathbb{Z}}$ be the ring of symmetric functions of bounded degrees with infinite variables $\x=(x_1,x_2,x_3,\ldots)$. 
It can be viewed as the polynomial ring $\mathbb{Z}[e_1,e_2,e_3,\ldots]=\mathbb{Z}[h_1,h_2,h_3,\ldots]$ where $e_n$ and $h_n$ are the \emph{elementary symmetric function} and the \emph{complete homogeneous symmetric function} of degree $n$ respectively. 
The ring $\Lambda_{\mathbb{Z}}=\bigoplus_{i\ge 0}\Lambda_i$, where $\Lambda_i$ is the space of symmetric functions homogeneous of degree $i$, forms a graded $\mathbb{Z}$-algebra via the usual multiplication.  

For any positive integer $n$, let $\C(\S_n)$ be the space of virtual characters of $\S_n$, i.e. the free $\mathbb{Z}$-module generated by the irreducible characters $\chi^{\lambda}$ of $\S_n$ for any partition $\lambda$ of $n$, denoted by $\lambda\vdash n$. The irreducible representation corresponding to the character $\chi^\lambda$ is called the \emph{Specht module $S^\lambda$} for any partition $\lambda$. 
Then the direct sum 
\[
    \C(\S)\coloneqq\bigoplus_{n\ge 0}\C(\S_n)
\]
forms a graded $\mathbb{Z}$-algebra under the \emph{induction product} defined by 
\begin{align*}
     *:\C(\S_{n_1})\times \C(\S_{n_2}) &\longrightarrow \C(\S_{n_1+n_2})\\
        (f, g) &\longmapsto f*g\coloneqq(f\otimes g)\uparrow_{\S_{n_1}\times\S_{n_2}}^{\S_{n_1+n_2}}
\end{align*}
where for a subgroup $H$ of $G$ the operation $(-)\uparrow_{H}^{G}$ is the usual induction from the class functions of $H$ to the class functions of $G$.

There is a $\mathbb{Z}$-algebra isomorphism between $\C(\S)$ and $\Lambda_{\mathbb{Z}}$ called the \emph{Frobenius characteristic map} $\ch$ such that
\begin{align*}
    \ch: \C(\S) &\longrightarrow \Lambda_{\mathbb{Z}}\\
    \mathbf{1}_{\S_n} & \longmapsto h_n\\
        {\rm sgn}_{\S_n}  & \longmapsto e_n\\
            {\chi}^{\lambda} & \longmapsto s_{\lambda}
\end{align*}
where $\mathbf{1}_{\S_n}$, ${\rm sgn}_{\S_n}$ are the trivial and signed characters of $\S_n$ respectively, and $s_{\lambda}$ is the \emph{Schur function} associated to $\lambda$. 
To simplify the notations, we sometimes denote the representation by the corresponding character. 
For a composition $\nu=(\nu_1,\nu_2,\ldots,\nu_{\ell})$ of $n$, let $\S_{\nu}$ be the corresponding Young subgroup. 
Then the $\S_n$-permutation module whose pointwise-stabilizer is $\S_{\nu}$ is isomorphic to the induced representation $\mathbf{1}_{\S_{\nu}}\uparrow_{\S_{\nu}}^{\S_n}$ and has the Frobenius characteristic
\[
    \ch(\mathbf{1}_{\S_{\nu}}\uparrow_{\S_{\nu}}^{\S_n})=\ch(1_{\S_{\nu_1}}*1_{\S_{\nu_2}}*\ldots*1_{\S_{\nu_{\ell}}})=h_{\nu_{1}}h_{\nu_{2}}\ldots h_{\nu_{\ell}}\eqqcolon h_{\nu}.
\]
Furthermore, the dimension of the permutation module is given by
\[
    \dim(\mathbf{1}_{\S_{\nu}}\uparrow_{\S_{\nu}}^{\S_n})=\dim(\mathbb{C}[\S_n/\S_{\nu}])=|\S_n:\S_{\nu}|=\frac{n!}{\nu_1!\ldots\nu_{\ell}!}=\binom{n}{\nu_1,\ldots,\nu_{\ell}}.
\]

Let $V=\bigoplus_i V^i$ be a graded $\mathbb{C}[\S_n]$-module. The \emph{graded Frobenius series} of $V$ is defined to be $\grFrob(V,t)\coloneqq \sum_{i}\ch(V^i)t^i$. In this way, one can encode representations of $\S_n$ in terms of symmetric functions via $\ch$ and this will be the method we apply to study the representations of $\S_n$ in this paper.

\subsection{Eulerian polynomials for derangements, permutations, and decorated permutations and their extensions} \label{subsec:Permutations}
\subsubsection{Eulerian story} \label{subsubsec:Eulerian}
For any permutation $\sigma=\sigma_1\sigma_2\ldots\sigma_n$ in $\S_n$, one can consider the following statistics:
\begin{itemize}
    \item \emph{Excedance set} $\EXC(\sigma)\coloneqq\{i\in[n-1]: \sigma_i>i\}$ and \emph{excedance number} $\exc(\sigma)\coloneqq |\EXC(\sigma)|$.
    \item \emph{Descent set} $\DES(\sigma)\coloneqq\{i\in[n-1]: \sigma_i> \sigma_{i+1}\}$ and \emph{descent number} $\des(\sigma)\coloneqq |\DES(\sigma)|$.
    \item \emph{Inversion number} $\inv(\sigma)\coloneqq |\{(i,j)\in[n]\times[n]: i<j,~\sigma_i> \sigma_j\}|$.
    \item \emph{Major index} $\maj(\sigma)\coloneqq \sum_{i\in\DES(\sigma)}i$.
\end{itemize}

It is well known that $\exc$ and $\des$ are equidistributed statistics on $\S_n$ and belong to the class of \emph{Eulerian statistics}. The generating function of their distribution 
\begin{equation} \label{eq:DefEulPoly}
    A_n(t)\coloneqq \sum_{\sigma\in\S_n}t^{\des(\sigma)}=\sum_{\sigma\in\S_n} t^{\exc(\sigma)}
\end{equation}
is known as the \emph{Eulerian polynomial} and its coefficients are called the \emph{Eulerian numbers}. 
Recall that $D_n$ is the subset of derangements in $\S_n$, i.e., permutations without fixed points.
Similarly, we can consider the analogous Eulerian polynomial for derangements
\[
    d_n(t)\coloneqq \sum_{\sigma\in D_n}t^{\exc(\sigma)}.
\]
Led by their study on Rees product of posets, Shareshian and Wachs \cite{ShareshianWachs2010} studied certain $q$-analogs and symmetric function analogs of the Eulerian polynomials $A_n(t)$ and $d_n(t)$, which we will introduce in the following.

\begin{defi}
Let $[\overline{n}]\coloneqq\{\overline{1},\overline{2},\ldots,\overline{n}\}$ and assign to $[n]\cup[\overline{n}]$ the total order $\overline{1}<\ldots<\overline{n}<1<\ldots<n$. Let $\overline{\sigma}=\overline{\sigma}_1\ldots\overline{\sigma}_n$ be a word over $[n]\cup[\overline{n}]$ obtained from $\sigma$ by letting $\overline{\sigma}_i=\overline{j}$ if $\sigma_i=j$ and $i\in\EXC(\sigma)$, letting $\overline{\sigma}_i=\sigma_i$ if $i\notin\EXC(\sigma)$. Define the set statistic
\[
	\DEX(\sigma)\coloneqq\DES(\overline{\sigma}).
\]    
\end{defi}
\noindent The statistic $\DEX$ is a variant of descent set $\DES$ and has intriguing properties that connect the classical statistics $\maj$ and $\exc$. 

\begin{thm} \label{thm:DEXPerm} \cite[Lemma 2.2]{ShareshianWachs2010} 
For all $\sigma \in \S_n$, 
\[
    \sum_{i\in\DEX(\sigma)}i=\maj(\sigma)-\exc(\sigma)
\]
and 
\[
    |\DEX(\sigma)|=\begin{cases}
        \des(\sigma)  & \text{ if }\sigma(1)=1\\
        \des(\sigma)-1 &  \text{ if }\sigma(1)\neq 1
    \end{cases}.
\]
\end{thm}

For an integer $n\ge 1$ and $S\subseteq[n-1]$, define Gessel's \emph{fundamental quasisymmetric function}
\[
    F_{S,n}(\x)\coloneqq \sum_{\substack{i_1\ge i_2\ge \ldots\ge i_n\\ j\in S\Rightarrow i_j>i_{j+1}}}x_{i_1}x_{i_2}\ldots x_{i_n}
\]
and $F_{\phi,0}\coloneqq 1$. In particular, when $S=\emptyset$, $F_{\emptyset, n}=h_n$; when $S=[n-1]$, $F_{[n-1],n}=e_n$. 

\begin{defi}\label{def:EulerianQuasi}
For integer $n\ge 1$, define 
\[
    Q_n(\x,t)=\sum_{\sigma\in\S_n}F_{\DEX(\sigma),n}(\x)t^{\exc(\sigma)}  
\]
and 
\[
    Q_n^0(\x,t)=\sum_{\sigma\in D_n}F_{\DEX(\sigma),n}(\x)t^{\exc(\sigma)}.
\]
It is useful to set $Q_0(\x,t)\coloneqq 1$. For $0\le j\le n-1$, we denote by $Q_{n,j}(\x)$ (resp. $Q^0_{n,j}(\x)$) the quasisymmetric function coefficient of $t^j$ in $Q_n(\x,t)$ (resp. $Q^0_n(\x,t)$. 
\end{defi}

Shareshian and Wachs showed that $Q_{n,j}(\x)$ and $Q^0_{n,j}(\x)$ are symmetric functions for all~ $j$.
The polynomials $Q_n(\x,t)$ and $Q^0_n(\x,t)$ can be viewed as symmetric-function analogs of the Eulerian polynomials $A_n(t)$ and $d_n(t)$, respectively. 
The functions $Q_{n,j}(\x)$ and $Q^0_{n,j}(\x)$ are the symmetric-function analogs of the Eulerian numbers and were referred to as the \emph{Eulerian quasisymmetric function} in \cite{ShareshianWachs2010}.


The symmetric function $Q_{n,j}(\x)$ is $h$-positive (has positive coefficients when expanded as a linear combination of $\{h_{\lambda}\}_{\lambda\vdash n}$) and appears as the Frobenius characteristic of several natural $\S_n$-modules, including the cohomology of the \emph{permutahedral variety} \cite{Stanley1989,Stembridge1992,Stembridge1994}, the Chow ring of the Boolean matroid $\mathsf{B}_n=U_{n,n}$, which is isomorphic (as a ring) to that cohomology \cite{Liao2023FPSAC, Liao2023One}, and the poset homology of Rees product of the Boolean lattice with certain tree-like posets \cite{ShareshianWachs2007,ShareshianWachs2012Rees,Athanasiadis2020Rees}. 
In particular, identifying $Q_{n,j}(\x)$ with the Frobenius characteristic of the permutation representation on the Chow ring $A(\mathsf{B}_n)$ yields the following theorem. 
\begin{thm}\label{thm:FY^jPermbasis} \cite[Proposition 3.5, Theorem 3.6]{Liao2023One} 
For $0\le j\le n-1$, the degree $j$ component of the Chow ring $A^j(\mathsf{B}_n)=\mathbb{C}\S_n FY^j(\mathsf{B}_n)$ is a permutation representation of $\S_n$ and
\[
    \ch(\mathbb{C}\S_n FY^j(\mathsf{B}_n))=Q_{n,j}(\x).
\]
\end{thm}

The \emph{stable principal specialization} of a (quasi)symmetric function $f(\x)$ is defined by $\ps(f)\coloneqq f(1,q,q^2,\ldots)$. The following lemma is implicit in the proof of Lemma 5.2 in \cite{Gessel1993cycle}.

\begin{lem}\label{Lem:psF} 
For all $n\ge 1$, $S\subseteq [n-1]$, we have 
\[
    \ps(F_{S,n})=\frac{q^{\sum_{i\in S}i}}{(1-q)(1-q^2)\ldots(1-q^n)}.
\]
\end{lem}
It follows from Lemma \ref{Lem:psF} that Theorem \ref{thm:DEXPerm} implies the following corollary.

\begin{cor} \label{cor:SumDEX} 
For all $\sigma\in\S_n$, 
\[
    \ps(F_{\DEX(\sigma),n})=\frac{q^{\maj(\sigma)-\exc(\sigma)}}{(1-q)(1-q^2)\ldots(1-q^n)}.
\]
\end{cor}

By Corollary \ref{cor:SumDEX}, the stable principal specialization of $Q_n(\x,t)$ and $Q_n^0(\x,t)$ gives the following interesting $q$-analogs of the Eulerian polynomials $A_n(t)$ and $d_n(t)$. 

\begin{thm}\label{Thm:psEulerianQ}
For all $n\ge 1$, 
\begin{enumerate}
    \item[(a)] $\prod_{i=1}^n\left(1-q^i\right)\ps(Q_n(\mathbf{x},t))=\sum_{\sigma\in\S_n}q^{\maj(\sigma)-\exc(\sigma)}t^{\exc(\sigma)}$,
    \item[(b)] $\prod_{i=1}^n\left(1-q^i\right)\ps(Q_n^0(\mathbf{x},t))=\sum_{\sigma\in D_n}q^{\maj(\sigma)-\exc(\sigma)}t^{\exc(\sigma)}$.
\end{enumerate}
\end{thm}
\noindent
Throughout this paper, we denote  
\begin{equation}\label{qEulerian}
	A_n(q,t)\coloneqq\sum_{\sigma\in\S_n}q^{\maj(\sigma)-\exc(\sigma)}t^{\exc(\sigma)},
\end{equation}
\[
    d_n(q,t)\coloneqq\sum_{\sigma\in D_n}q^{\maj(\sigma)-\exc(\sigma)}t^{\exc(\sigma)}.
\]

\subsubsection{Binomial Eulerian story} \label{subsubsec:BinomialEulerian}
The \emph{binomial Eulerian polynomial}
\[
  \widetilde{A}_n(t)\;\coloneqq\;1 \;+\; t\sum_{k=1}^n \binom{n}{k} A_k(t),
\]
named by Shareshian and Wachs \cite{ShareshianWachs2020}, shares many interesting properties with the classical Eulerian polynomial $A_n(t)$.  In particular, it is palindromic and also arises as the $h$-polynomial of a simple polytope called the \emph{stellohedron}. This polytopal interpretation was first uncovered by Postnikov, Reiner, and Williams \cite{PRW2008}. The palindromicity of $\widetilde{A}_n(t)$ is equivalent to an intriguing symmetrical Eulerian identity studied by Chung, Graham, and Knuth \cite{CGKEulerianSum2011}.  

A natural $q$-analog and a symmetric‑function analogue of $\widetilde{A}_n(t)$ were introduced by Shareshian and Wachs \cite{ShareshianWachs2020}:
\begin{equation}
    \widetilde{A}_n(q,t)\coloneqq 1+t\sum_{k=1}^n\qbin{n}{k}A_k(q,t),
\end{equation}
\begin{equation} \label{defBinomQ}
	\widetilde{Q}_n(\x,t)\coloneqq h_n(\x)+t\sum_{k=1}^n h_{n-k}(\x)Q_k(\x,t).
\end{equation}
Just as $Q_n(\x,t)$ is the graded Frobenius series of the natural $\S_n$-representation on the cohomology ring of the permutahedral variety, Shareshian and Wachs proved that $\widetilde{Q}_n(\x,t)$ is the graded Frobenius series of the cohomology of the stellahedral variety, which, by definition (\ref{defBinomQ}), carries a permutation representation of $\S_n$. 
This cohomology ring is isomorphic to the augmented Chow ring of the Boolean matroid. In particular, one obtains the following.

\begin{thm} \cite[Theorems 4.6, 4.8]{Liao2023One}\label{thm:TildeFY^jPermbasis} 
For $0\le j\le n$, the degree $j$ component of the augmented Chow ring $\widetilde{A}^j(\mathsf{B}_n)=\mathbb{C}\S_n \widetilde{FY}^j(\mathsf{B}_n)$ is a permutation representation of $\S_n$ and
\[
    \ch(\mathbb{C}\S_n \widetilde{FY}^j(\mathsf{B}_n))=\widetilde{Q}_{n,j}(\x),
\]
where $\widetilde{Q}_{n,j}(\x)$ is the coefficient of $t^j$ of $\widetilde{Q}_n(\x,t)$.
\end{thm}

Shareshian and Wachs further showed that $\widetilde{A}_n(q,t)$ is the appropriate $q$-analog of $\widetilde{A}_n(t)$ to consider:
\begin{equation}  \label{eq:PSBinomialQ}
    \prod_{i=1}^n(1-q^i)\ps(\widetilde{Q}_n(\x,t))=\widetilde{A}_n(q,t).
\end{equation}
As an analog of the set of permutations $\mathfrak{S}_n$ in the Eulerian story, the author \cite{Liao2023One} found that Postnikov's \emph{decorated permutations} give a nice combinatorial model to interpret the binomial Eulerian story. 
A \emph{decorated permutation} of $[n]$ can be viewed as a permutation on a subset of $[n]$.
We denote by $\widetilde{\mathfrak{S}}_n$ the set of decorated permutations of $[n]$.
\begin{exa}
For example, $215\in\widetilde{\mathfrak{S}_5}$ is a permutation on $\{1,2,5\}\subset[5]$ that maps $1$ to $2$, $2$ to $1$, $5$ to $5$. We can express it in \emph{two-line notation} and in \emph{one-line notation} respectively as the following 
\[
	215=\left(\begin{array}{ccccc}
		1 & 2 & 3 & 4 & 5\\
		2 & 1 & 0 & 0 & 5
	\end{array}\right)=21005.
\]
For each $n$, let $\theta$ be the permutation on $\emptyset$ in $\widetilde{\mathfrak{S}_n}$. In particular when $n=5$, it has the following two-line notation and one-line notation
\[
	\theta\coloneqq\left(\begin{array}{ccccc}
		1 & 2 & 3 & 4 & 5\\
		0 & 0 & 0 & 0 & 0
	\end{array}\right)=00000.
\]
\end{exa}
\begin{rem}
 Our definition of decorated permutations is different from Postnikov's in \cite{Postnikov2006Positroid} where he introduced them as permutations with two kinds of fixed points. An easy bijection is given by replacing the $0$s in our one-line notation with the second kind of fixed points.
\end{rem}
Extend the set statistic $\DEX$ over $\S_n$ to $\widetilde{\S}_n$ in the following way.

\begin{defi} \label{def:DEXext} \cite[Definition 6.19]{Liao2023One}
Consider the set $[n]\cup[\overline{n}]\cup\{0\}$ with the total order $\overline{1}<\ldots<\overline{n}<0<1<\ldots<n$. For $\sigma\in\widetilde{\S}_n\setminus\{\theta\}$, let $\overline{\sigma}$ be a word over $[n]\cup[\overline{n}]\cup\{0\}$ obtained from $\sigma$ by replacing $\sigma_i$ with $\overline{\sigma_i}$ whenever $i\in\EXC(\sigma)\coloneqq\{i\in[n]:i<\sigma_i\}$. Define the set statistic
\[
    \DEX(\sigma)\coloneqq\begin{cases}
        \DES(\overline{\sigma}) & \mbox{ if }\sigma\neq \theta\\
        \emptyset & \mbox{ if }\sigma=\theta 
    \end{cases}.
\]
\end{defi}

Similarly, the definition of the \emph{excedance number} and the \emph{major index} can be extended to $\sigma\in\widetilde{\S}_n$ as
\[
	\exc(\sigma)\coloneqq\begin{cases}
	|\{i\in[n]: i<\sigma_i\}| & \text{ if }\sigma\neq\theta\\
	-1  & \text{ if }\sigma=\theta
	\end{cases} ~\text{ and }~
	\maj(\sigma)\coloneqq\begin{cases}
	\sum_{i:\sigma_i<\sigma_{i+1}} i & \text{ if }\sigma\neq\theta\\
	-1  & \text{ if }\sigma=\theta
	\end{cases}.
\]
Then the extension of $\DEX$ to $\widetilde{\S}_n$ still preserves the nice property it has as a statistic over $\S_n$. The following lemma is analogous to Theorem \ref{thm:DEXPerm}.
\begin{prop} \label{Prop:ExtDEXSum} \cite[Lemma 6.20]{Liao2023One}
For $\sigma\in\widetilde{\S}_n$, 
\[
    \sum_{i\in\DEX(\sigma)}i=\maj(\sigma)-\exc(\sigma).
\]
And for $\sigma\in\widetilde{\S}_n\setminus\{\theta\}$,
\[
    |\DEX(\sigma)|=\begin{cases}
        \des(\sigma) & \mbox{ if }\sigma(1)=0 \text{ or }1\\
        \des(\sigma)-1 & \mbox{ if }\sigma(1)>1
    \end{cases}.
\]
\end{prop}
Liao used the extension of $\DEX$ to $\widetilde{\S}_n$ to obtain the $F$-expansion of $\widetilde{Q}_n(\x,t)$ analogous to Definition \ref{def:EulerianQuasi}.

\begin{thm} \label{thm:FexpWideQ} \cite[Theorem 6.21]{Liao2023One}
For any integer $n\ge 1$, 
\[
	\widetilde{Q}_n(\x,t)=\sum_{\sigma\in\widetilde{\mathfrak{S}}_n}F_{\DEX(\sigma),n}(\x)t^{\exc(\sigma)+1}.
\]
Or equivalently, for all $0\le j\le n$,
\[
    \widetilde{Q}_{n,j}(\x)=\sum_{\substack{\sigma\in\widetilde{\S}_n\\ \exc(\sigma)=j-1}} F_{\DEX(\sigma),n}(\x).
\]
\end{thm}
Combining Lemma \ref{Lem:psF}, (\ref{eq:PSBinomialQ}) and Proposition \ref{Prop:ExtDEXSum} gives the following Theorem analogous to Theorem \ref{Thm:psEulerianQ}.
\begin{thm} \label{thm:QDecoratedPerm}
For all integer $n\ge 1$, 
\[
    \prod_{i=1}^n(1-q^i)\ps(\widetilde{Q}_n(\x,t))=\sum_{\sigma\in\widetilde{\mathfrak{S}}_n}q^{\maj(\sigma)-\exc(\sigma)}t^{\exc(\sigma)+1}.
\]
\end{thm}

We summarize the relations between the combinatorial models and their enumerators lifted to different extensions discussed in this subsection in Table \ref{table:relations}.

\begin{table}[h]
\[
    \begin{array}{c|ccccc}
           & \substack{\text{enumerator}\\ \exc} &  & \substack{q-\text{enumerator}\\ (\maj-\exc,\exc)} &  &  \substack{\text{equivariant enumerator}\\ (\DEX,\exc)}\\
    \hline
    D_n   &   d_n(t)     &   &   d_n(q,t)  &  &   Q_n^0(\x,t)\\
    \cap \\
    \S_n  &   A_n(t)     & \xleftarrow{q=1}   &   A_n(q,t)  & \xleftarrow{\ps} &   Q_n(\x,t)  \\
    \cap \\
    \widetilde{\S}_n  & \widetilde{A}_n(t) &  & \widetilde{A}_n(q,t)  &  & \widetilde{Q}_n(\x,t) 
    \end{array}
\]
\caption{}
\label{table:relations}
\end{table}

\subsection{Group actions on rank-selected posets} \label{subsec:RankSelection}
In this section, we recall some terminology and facts about rank-selected posets from \cite{Stanley1982GroupPoset} and \cite{WachsTool2007}. 

Let $P$ be a finite poset with a unique minimal element $\hat{0}$ and a unique maximal element $\hat{1}$. If $P$ is graded of rank $n$ with rank function $\rk_P: P\longrightarrow\{0,1,\ldots,n\}$, then for $S\subseteq [n-1]$, we can define the \emph{rank-selected subposet} of $P$ by
\[
    P_S\coloneqq\{x\in P~:~\rk_P(x)\in S\}\cup\{\hat{0},\hat{1}\}.
\]
Let $G$ be a group of automorphisms of $P$. The action of $G$ preserves the rank of the elements in $P$; hence for any $S\subseteq [n-1]$, the group $G$ permutes the maximal chains in $P_S$. Let us denote $\alpha_P(S)$ the permutation representation of $G$ generated by the maximal chains in $P_S$. Consider the virtual representation
\[
    \beta_P(S)\coloneqq\sum_{T\subseteq S}(-1)^{|S|-|T|}\alpha_P(T).
\]
By M\"{o}bius inversion, we also have
\[
    \alpha_P(S)=\sum_{T\subseteq S}\beta_P(T).
\]

\begin{thm} \cite[Theorem 1.2]{Stanley1982GroupPoset} \label{thm:BetaGenRep}
If $P$ is Cohen-Macaulay, then $\beta_P(S)$ is a genuine representation of $G$ and 
\[
    \beta_P(S)\cong_G \tilde{H}_{|S|-1}(P_S).
\]
\end{thm}
Throughout this paper, the poset $P$ we care about is always the lattice of flats of a matroid. It is well known that the lattice of flats of a matroid is always Cohen-Macaulay (see \cite[Theorem 4.1]{Folkman1966} and \cite[p.14]{BjornerGarsiaStanley1982}).

When $P$ is the Boolean lattice $B_n$ and $G$ is the symmetric group $\S_n$, the structure of representation $\beta_{B_n}(S)\cong \tilde{H}_{|S|-1}({(B_n)}_S)$ is known for any selected rank set $S\subseteq [n-1]$. Below, we follow the notations from \cite[chapter 2]{WachsTool2007}. 

Let $\lambda$, $\mu$ be two partitions such that $\mu\subseteq\lambda$ (i.e. $\mu_i\le \lambda_i$ for all $i$). The connected skew shape $\lambda/\mu$ is said to be a \emph{ribbon} (or a \emph{skew hook} or a \emph{border strip}) if two consecutive rows always overlap in exactly one cell. Such a skew shape can be described by its number of cells and its \emph{descent set}. Let $H$ be a ribbon of $n$ cells. Number the cells of $H$ from $1$ to $n$, filling left to right in each row and filling the rows from bottommost to topmost. We call a cell $i$ a \emph{descent} of $H$ if cell $i+1$ is directly above cell $i$. The descent set of $H$ is the set of all descents of $H$. If the descent set of $H$ is $R\subseteq [n-1]$, then we denote $H$ by $H_{R,n}$.

\begin{exa}
Let $n=7$. Then the diagrams of $H_{\{2,4\},7}$ and $H_{\{1,3\},7}$ are 
\[
\begin{array}{ccccc}
    H_{\{2,4\},7} &   &    &    & H_{\{1,3\},7} \\
                & &      &  &            \\
    \begin{ytableau}
    \none & \none & 5 & 6 & 7\\
    \none & 3 & *(green)4\\ 
    1     & *(green)2 
    \end{ytableau}  
    &   &     &     &
    \begin{ytableau}
    \none & 4 & 5 & 6 & 7\\
    2     & *(green)3 \\
    *(green) 1
    \end{ytableau}
\end{array}.
\]
\end{exa}
Let $S^{\lambda/\mu}$ be the \emph{Specht module} of skew shape $\lambda/\mu$. If $\lambda/\mu$ is a ribbon, the corresponding Specht module is known as a \emph{Foulkes representation}. The following result was first obtained by Solomon \cite{Solomon1968} under different terminology; see also \cite[Theorem 3.4.4]{WachsTool2007}.

\begin{thm}\label{thm:RankHomologyBn}
Let $R\subseteq [n-1]$. Then
\[
    \tilde{H}_{|R|-1}({(B_n)}_R)\cong_{\S_n} S^{H_{R,n}}.
\]    
\end{thm}

\section{Graded Frobenius series for uniform matroids $U_{r,n}$} 

\subsection{Chow rings of uniform matroids} \label{Sec:ProofFrobChow}

In this section, we prove Theorem \ref{thm:FrobChowUniform}, which are our formulas for the Frobenius series of $A(U_{r,n})_{\mathbb{C}}$. Our proof of Theorem \ref{thm:FrobChowUniform} is parallel to Hameister, Rao, and Simpson's proof of (\ref{HRS:ChowQUniform}) in \cite{HameisterRaoSimpson2021}. 
We proceed by calculating the difference of the graded Frobenius series of $A(U_{r,n})$ between two consecutive rank $r$,
\begin{align*}
    \Delta_{n,r}(\x)
    &\coloneqq\grFrob(A(U_{r+1,n})_{\mathbb{C}},t)-\grFrob(A(U_{r,n})_{\mathbb{C}},t)\\
    &=a_{n,r}^{(1)}(\x)t+a_{n,r}^{(2)}(\x)t^2+\cdots+a_{n,r}^{(r-1)}(\x)t^{r-1}+a_{n,r}^{(r)}t^{r}
\end{align*}  
for $1\le r\le n-1$. In Proposition \ref{prop:DifferenceChow}, we show that $\Delta_{n,r}(\x)$ can be expressed nicely in two different ways. Then the graded Frobenuis series of $A(U_{r,n})_{\mathbb{C}}$ can be obtained by considering either
\[
    \grFrob(A(U_{r,n})_{\mathbb{C}},t)
    =\grFrob(A(U_{1,n})_{\mathbb{C}},t)+\sum_{j=1}^{r-1}\Delta_{n,j}(\x)
\]
or 
\[
    \grFrob(A(U_{r,n})_{\mathbb{C}},t)
    =\grFrob(A(U_{n,n})_{\mathbb{C}},t)-\sum_{j=r}^{n-1}\Delta_{n,j}(\x),
\]
which with a proper choice of expression of $\Delta_{n,j}(\x)$ leads to Theorem \ref{thm:FrobChowUniform} (i) and (ii) respectively.

Before proving our main results, we collect several facts that will be useful for expressing $\Delta_{n,r}(\x)$. 
In Theorem \ref{thm:FY^jPermbasis}, we saw that the graded Frobenius series of the Chow ring of the Boolean matroid coincides with the Eulerian quasisymmetric function
\[
    \grFrob(A(\mathsf{B}_n)_{\mathbb{C}},t)=Q_n(\x,t)=\sum_{\sigma\in\S_n}F_{\DEX(\sigma),n}(\x)t^{\exc(\sigma)},
\]
whose generating function admits a closed form (\cite[Theorem 1.2]{ShareshianWachs2010}):
\begin{equation} \label{eq:GFEulQuasi}
    1+\sum_{n\ge 1}Q_{n}(\x,t)z^n=\frac{(1-t)H(z)}{H(tz)-tH(z)},
\end{equation}
where $H(z)=\sum_{n\ge 0}h_n(\x)z^n$. 
In fact, Shareshian and Wachs further refined this by tracking fixed points. 
They defined
\[
    Q_{n,j,k}(\x)\coloneqq\sum_{\substack{\sigma\in\S_n\\ \exc(\sigma)=j,~ \fix(\sigma)=k}}F_{\DEX(\sigma),n}(\x)
\]
and proved that 
\[
    1+\sum_{n\ge 1,j,k\ge 0}Q_{n,j,k}(\x)t^j r^k z^n=\frac{(1-t)H(rz)}{H(tz)-tH(z)}.
\]
We give a refinement of Theorem \ref{thm:FY^jPermbasis} regarding $Q_{n,j,k}(\x)$ that will be used to express $\Delta_{n,r}(\x)$.
Let the subset 
\begin{equation} \label{eq:FYjk}
    FY^{j,k}(\mathsf{B}_n)\coloneqq\left\{ x_{F_1}^{a_1}\ldots x_{F_\ell}^{a_\ell}\in FY^j(\mathsf{B}_n): |[n]-F_\ell|=k\right\}
\end{equation}
and $A^{j,k}(\mathsf{B}_n)$ be the subspace spanned by $FY^{j,k}(\mathsf{B}_n)$, which is obviously $\S_n$-invariant.

\begin{prop} \label{prop:RefinedZeroCodes}
For $0\le j\le n-1$ and $0\le k\le n$, we have 
\[
\ch(A^{j,k}(\mathsf{B}_n)_{\mathbb{C}})=Q_{n,j,k}(\x).
\]
\end{prop}
\begin{proof}
To prove the result, we use combinatorial objects only needed in this proof called \emph{Stembridge codes}. The readers are referred to Section 3.1 and Theorem 3.6 in \cite{Liao2023One} for their definition and an $\S_n$-equivariant bijection between Stembridge codes and $FY(\mathsf{B}_n)$. Stembridge \cite[Lemma 3.1]{Stembridge1992} showed that the set $\Code_{n,j}$ of Stembridge codes with index $j$ generated a $\mathbb{C}\S_n$-module satisfying
\[
    1+\sum_{n\ge 1}\sum_{j=0}^{n-1}\ch(\mathbb{C}\Code_{n,j}) t^j z^n=\frac{(1-t)H(z)}{H(tz)-tH(z)}.
\]
Stembridge's calculation actually extends without effort to
\[
    1+\sum_{n\ge 1 ,j,k\ge 0}\ch\left(\mathbb{C}\Code_{n,j,k}\right)t^j r^k z^n=\frac{(1-t)H(rz)}{H(tz)-tH(z)},
\]
or equivalently, $\ch(\mathbb{C}\Code_{n,j,k})=Q_{n,j,k}(\x)$, where $\Code_{n,j,k}$ is the subset of Stembridge codes of index $j$ and $k$ zeros.
Consider the $\S_n$-equivariant bijection $\phi$ between $FY(\mathsf{B}_n)$ and Stembridge codes in \cite[Theorem 3.6]{Liao2023One}. The restriction of $\phi$ to $FY^{j,k}(\mathsf{B}_n)$ gives an $\S_n$-equivariant bijection $FY^{j,k}(\mathsf{B}_n)\longrightarrow\Code_{n,j,k}$. Hence we have
\[
    \ch(A^{j,k}(\mathsf{B}_n))=\ch(\mathbb{C}\Code_{n,j,k})=Q_{n,j,k}(\x).    
\]
\end{proof}

We need the following two properties of $Q_{n,j,k}$. The first one is about the palindromicity of $Q_{n,j,k}$ and is a special case of \cite[Theorem 5.9]{ShareshianWachs2010}.

\begin{lem} \label{lem:PalinQ}
For positive integers $n$, $0\le k \le n$, and $0\le j\le n-k$,
\[
    Q_{n,j,k}(\x)=Q_{n,n-k-j,k}(\x).
\]
\end{lem}

\begin{lem} \cite[Corollary 3.4]{ShareshianWachs2010} \label{lem:ExtractHQ}
For all $n,j,k$, 
\[
    Q_{n,j,k}(\x)=h_kQ_{n-k,j,0}(\x).
\]
\end{lem}

Now we have the following interpretations of the difference $\Delta_{n,r}(\x)$ between $\grFrob(A(U_{r+1,n}),t)$ and $\grFrob(A(U_{r,n}),t)$.

\begin{prop} \label{prop:DifferenceChow}
For $1\le r\le n-1$, the polynomial $\Delta_{n,r}(\x)\in\Lambda_n[t]$ has two expressions:
\[
   \sum_{i=0}^r h_{n-r+i}Q_{r-i}^0(\x,t)t^i
\]
and
\[
    \sum_{\substack{\sigma\in\S_n \\ \fix(\sigma)\ge n-r}}F_{\DEX(\sigma),n}(\x)t^{r-\exc(\sigma)}.
\]
\end{prop}

\begin{proof}
Note that for $0\le k\le n-1$, as sets we have
\[
    FY^k(U_{1,n})\subseteq FY^k(U_{2,n})\subseteq\ldots\subseteq FY^k(U_{n,n}).
\]
The coefficient $a_{n,r}^{(k)}(\x)$ of $t^k$ in $\Delta_{n,r}(\x)$ is the Frobenius characteristic of the permutation representation generated by $FY^k(U_{r+1,n})\setminus FY^k(U_{r,n})$, which consists of the monomial
\[
    x_{[n]}^r \quad \text{ if }k=r
\]
and monomials
\[
    x_{F_1}^{a_1}\ldots x_{F_\ell}^{a_{\ell}}x_{[n]}^i\in FY^k(U_{r+1,n})  \quad\text{such that}\quad i=(r+1)-|F_\ell|-1=r-|F_\ell|
\]
for $i=0,1,\ldots,k-1$ if $1\le k\le r-1$. Hence $a_{n,r}^{(r)}(\x)=h_n$; and for $1\le k\le r-1$, one can construct an easy $\S_n$-equivariant bijection $x_{F_1}^{a_1}\ldots x_{F_\ell}^{a_{\ell}}x_{[n]}^i\mapsto x_{F_1}^{a_1}\ldots x_{F_\ell}^{a_{\ell}}$ that maps $FY^k(U_{r+1,n})\setminus FY^k(U_{r,n})$ to the set of monomials $x_{F_1}^{a_1}\ldots x_{F_\ell}^{a_{\ell}}$ such that $F_1\subsetneq\ldots\subsetneq F_{\ell}$ and $a_1,\ldots,a_\ell$ satisfy
\begin{itemize}
    \item $|F_\ell|=r-i$,
    \item $a_j\le |F_j|-|F_{j-1}|-1$ for all $j$,
    \item $a_1+\ldots+a_\ell=k-i$,
\end{itemize}
for $i=0,1,\ldots, k-1$. Recalling our notation in (\ref{eq:FYjk}), the above set of monomials can be viewed as $\bigcup_{i=0}^{k-1} FY^{k-i,n-r+i}(\mathsf{B}_n)$; since $FY^{k-i,n-r+i}(\mathsf{B}_n)$ is a permutation basis for $A^{k-i,n-r+i}(\mathsf{B}_n)_{\mathbb{C}}$, by Proposition \ref{prop:RefinedZeroCodes}, the coefficient of $t^k$ in $\Delta_{n,r}(\x)$ is
\begin{equation}
    a_{n,r}^{(k)}(\x)=\sum_{i=0}^{k-1} Q_{n,k-i,n-r+i}.  \label{eq:hQzero}
\end{equation}
Then
\begin{align*}
    \Delta_{n,r}(\x)
    &=\sum_{k=1}^r a_{n,r}^{(k)}(\x)t^k=h_nt^r+\sum_{k=1}^{r-1}\sum_{i=0}^{k-1} h_{n-r+i}Q_{r-i,k-i,0}t^k \quad(\text{by Lemma \ref{lem:ExtractHQ}})\\
    &=h_nt^r+\sum_{i=0}^{r-2} h_{n-r+i}\left(\sum_{k=i+1}^{r-1} Q_{r-i,k-i,0}t^{k-i}\right)t^i
    =h_nt^r+\sum_{i=0}^{r-2} h_{n-r+i}Q_{r-i}^0(\x,t)t^i\\
    &=\sum_{i=0}^{r} h_{n-r+i}Q_{r-i}^0(\x,t)t^i  \quad (\text{since }Q_1^0(\x,t)=0).
\end{align*}
Note that $\Delta_{n,1}(\x)=h_n t$.

On the other hand, we can obtain a different expression from (\ref{eq:hQzero}) as follows
\begin{align*}
    \Delta_{n,r}(\x)
     &=h_nt^r+\sum_{k=1}^{r-1}\sum_{i=0}^{k-1} Q_{n,k-i,n-r+i}t^k\\
     &=h_nt^r+\sum_{k=1}^{r-1}\sum_{i=0}^{k-1} Q_{n,r-k,n-r+i}t^k \quad(\text{by Lemma \ref{lem:PalinQ}})\\
     &=h_nt^r+\sum_{k=1}^{r-1}\sum_{\substack{\sigma\in\S_n \\ \exc(\sigma)=r-k \\ \fix(\sigma)\ge n-r}} F_{\DEX(\sigma),n}(\x)t^k\\
     &=F_{\DEX(e),n}(\x)t^{r} +\sum_{\substack{\sigma\in\S_n\setminus\{e\} \\ \fix(\sigma)\ge n-r}}F_{\DEX(\sigma),n}(\x)t^{r-\exc(\sigma)}\\
     &=\sum_{\substack{\sigma\in\S_n \\ \fix(\sigma)\ge n-r}}F_{\DEX(\sigma),n}(\x)t^{r-\exc(\sigma)}
\end{align*}
where $e$ is the identity in $\S_n$.
\end{proof}

We are ready to prove Theorem \ref{thm:FrobChowUniform}, which we restate here.

\begin{thmn}[\ref{thm:FrobChowUniform}] 
The graded Frobenius series of the Chow ring of arbitrary uniform matroids $\grFrob(A(U_{r,n})_{\mathbb{C}},t)$ has two expressions:
\begin{itemize}
    \item[(i)] \(\displaystyle \sum_{j=0}^{r-1}h_{n-j}Q_j^0(\x,t)(1+t+\ldots+t^{r-j-1})\),
    \item[(ii)] \(\displaystyle Q_n(\x,t)-\sum_{j=r}^{n-1}\sum_{\substack{\sigma\in\S_n\\ \fix(\sigma)\ge n-j}}F_{\DEX(\sigma),n}t^{j-\exc(\sigma)}\).
\end{itemize}
\end{thmn}

\begin{proof}
To prove expression (i), we apply the first expression from Proposition \ref{prop:DifferenceChow} and obtain
\begin{align*}
    \grFrob(A(U_{r,n})_{\mathbb{C}},t)
    &=\grFrob(A(U_{1,n})_{\mathbb{C}},t)+\sum_{j=1}^{r-1}\Delta_{n,j}(\x)=h_n+\sum_{j=1}^{r-1}\Delta_{n,j}(\x)\\
    &=h_n+\sum_{j=1}^{r-1}\sum_{i=0}^j h_{n-j+i}Q_{j-i}^0(\x,t)t^i
    =h_n+\sum_{j=1}^{r-1}\sum_{i=0}^j h_{n-i}Q_{i}^0(\x,t)t^{j-i}\\
    &=h_n+h_n(t+t^2+\ldots+t^{r-1})+\sum_{j=1}^{r-1}\sum_{i=1}^j h_{n-i}Q_{i}^0(\x,t)t^{j-i}\\
    &=h_n(1+t+\ldots+t^{r-1})+\sum_{i=1}^{r-1}\sum_{j=i}^{r-1}h_{n-i}Q_{i}^0(\x,t)t^{j-i}\\
    &=\sum_{i=0}^{r-1}h_{n-i}Q_{i}^0(\x,t)(1+t+\ldots+t^{r-1-i}).
\end{align*}
To prove expression (ii), we apply the second expression from Proposition \ref{prop:DifferenceChow} and obtain
\begin{align*}
    \grFrob(A(U_{r,n})_{\mathbb{C}},t)
    &=\grFrob(A(U_{n,n})_{\mathbb{C}},t)-\sum_{j=r}^{n-1}\Delta_{n,j}(\x)\\
    &=Q_n(\x,t)-\sum_{j=r}^{n-1}\sum_{\substack{\sigma\in\S_n\\ \fix(\sigma)\ge n-j}}F_{\DEX(\sigma),n}t^{j-\exc(\sigma)}.
\end{align*}
\end{proof}

Now Corollary \ref{cor:FrobChowSpecial}, which we restate here, can be proved directly from Theorem \ref{thm:FrobChowUniform}.

\begin{corn}[\ref{cor:FrobChowSpecial}]
For the special case that $r=n$ and $r=n-1$ we have
\begin{itemize}
    \item[(i)] \(\displaystyle \grFrob(A(U_{n,n})_{\mathbb{C}},t)=Q_n(\x,t).\)
    \item[(ii)]\(\displaystyle \grFrob(A(U_{n-1,n})_{\mathbb{C}},t)=t^{-1}Q_n^0(\x,t).\)
\end{itemize}
\end{corn}

\begin{proof}
When $r=n$, the matroid is Boolean. See Remark \ref{rem:BooleanIsProved}. When $r=n-1$, by Theorem \ref{thm:FrobChowUniform} (ii), we have
\begin{align*}
    \grFrob(A(U_{r,n})_{\mathbb{C}},t)
    &=Q_n(\x,t)-\sum_{\substack{\sigma\in\S_n \\ \fix(\sigma)\ge 1}}F_{\DEX(\sigma),n}t^{n-1-\exc(\sigma)}.
\end{align*}
Since $Q_{n}(\x,t)=\sum_{j=0}^{n-1}Q_{n,j}(\x)t^{j}$ and $Q_{n,j}=Q_{n,n-1-j}$ by Lemma \ref{lem:PalinQ}, we have
\begin{align*}
    \grFrob(A(U_{n-1,n})_{\mathbb{C}},t)
    &=\sum_{\sigma\in\S_n}F_{\DEX(\sigma),n}t^{n-1-\exc(\sigma)}-\sum_{\substack{\sigma\in\S_n \\ \fix(\sigma)\ge 1}}F_{\DEX(\sigma),n}t^{n-1-\exc(\sigma)}\\
    &=\sum_{\substack{\sigma\in\S_n \\ \fix(\sigma)=0}}F_{\DEX(\sigma),n}t^{n-1-\exc(\sigma)}
    =\sum_{j=1}^{n-1}Q_{n,j,0}(\x)t^{n-1-j}\\
    &=\sum_{j=1}^{n-1}Q_{n,n-j,0}(\x)t^{n-1-j} \quad \text{(by Lemma \ref{lem:PalinQ})}\\
    &=t^{-1}Q_n^0(\x,t).
\end{align*}
\end{proof}

\subsection{Augmented Chow rings of uniform matroids} \label{Sec:ProofFrobAug}
In this section, our goal is to prove Theorem \ref{thm:FrobAugUniform}, the augmented analog of Theorem \ref{thm:FrobChowUniform}. Similar to the proof of the non-augmented case, we will proceed by calculating the difference between the graded Frobenius series of the augmented Chow ring of $U_{r,n}$ for consecutive $r$, that is
\begin{align*}
    \widetilde{\Delta}_{n,r}(\x)
    &\coloneqq \grFrob(\widetilde{A}(U_{r+1,n})_{\mathbb{C}},t)-\grFrob(\widetilde{A}(U_{r,n})_{\mathbb{C}},t)\\
    &=\widetilde{a}_{n,r}^{(1)}(\x)t+\widetilde{a}_{n,r}^{(2)}(\x)t^2+\cdots+\widetilde{a}_{n,r}^{(r+1)}(\x)t^{r+1}.
\end{align*}
We give two expressions for $\widetilde{\Delta}_{n,r}(\x)$ in Proposition \ref{prop:DifferenceAug}. Then we derive the two formulas in Theorem \ref{thm:FrobAugUniform} from these two expressions by considering 
\[
    \grFrob(\widetilde{A}(U_{r,n})_{\mathbb{C}},t)
    =\grFrob(\widetilde{A}(U_{1,n}),t)+\sum_{j=1}^{r-1}\widetilde{\Delta}_{n,j}(\x)
\]
and 
\[
    \grFrob(\widetilde{A}(U_{r,n})_{\mathbb{C}},t)
    =\grFrob(\widetilde{A}(U_{n,n}),t)-\sum_{j=r}^{n-1}\widetilde{\Delta}_{n,j}(\x).
\]

We first establish some facts that will be useful in expressing $\widetilde{\Delta}_{n,r}(\x)$. These facts are parallel to what we did in the non-augmented case. In Theorem \ref{thm:TildeFY^jPermbasis}, we saw that the graded Frobenius series of the augmented Chow ring of the Boolean matroid is given by the Binomial Eulerian quasisymmetric function 
\[
    \grFrob\left(\widetilde{A}(\mathsf{B}_n)_{\mathbb{C}},t\right)=\widetilde{Q}_n(\x,t)=\sum_{\sigma\in\widetilde{\S}_n}F_{\DEX(\sigma),n}(\x)t^{\exc(\sigma)+1}.
\]
We want an analog of $Q_{n,j,k}(\x)$ for the binomial Eulerian quasisymmetric function $\widetilde{Q}_n(\x,t)$. 
Define the following variant of the binomial Eulerian quasisymmetric function
\[
    \widetilde{Q}_n(\x,t,r)\coloneqq h_n r^n+t\sum_{k=1}^n h_{n-k}Q_k(\x,t)r^{n-k}.
\]
Then $\widetilde{Q}_n(\x,t,1)=\widetilde{Q}_{n}(\x,t)$. Write $\widetilde{Q}_n(\x,t,r)=\sum_{j\ge 0, k\ge 0}\widetilde{Q}_{n,j,k}(\x)t^j r^k$.

\begin{thm} \label{thm:binEulerFix}
For any integer $n\ge 1$, 
\[
	\widetilde{Q}_n(\x,t,r)=\sum_{\sigma\in\widetilde{\mathfrak{S}}_n}F_{\DEX(\sigma),n}(\x)t^{\exc(\sigma)+1}r^{\fixTwo(\sigma)}.
\]
Or equivalently, for all $0\le j,~k\le n$,
\[
    \widetilde{Q}_{n,j,k}(\x)=\sum_{\substack{\sigma\in\widetilde{\S}_n\\ \exc(\sigma)=j-1\\ \fixTwo(\sigma)=k}} F_{\DEX(\sigma),n}(\x).
\]
In particular, 
\[
    \widetilde{Q}_{n,j,0}(\x)=\sum_{\substack{\sigma\in\S_n \\ \exc(\sigma)=j-1}}F_{\DEX(\sigma),n}(\x)=Q_{n,j-1}(\x).
\]
\end{thm}

\begin{proof} 
The proof is similar to the proof of Theorem 6.21 in \cite{Liao2023One}. By definition of $\widetilde{Q}_n(\x,t,r)$ and $h_n=F_{\emptyset,n}(\x)$, we have
\begin{align*}
    \widetilde{Q}_n(\x,t,r)
    &=F_{\emptyset,n}(\x)r^n+\sum_{k=1}^n\sum_{\sigma\in\S_k}F_{\emptyset,n-k}(\x)F_{\DEX(\sigma),k}(\x)t^{\exc(\sigma)+1}r^{n-k}\\
    &=F_{\emptyset,n}(\x)r^n+\sum_{k=1}^n\sum_{\substack{\pi\in\widetilde{\S}_n\\ \fixTwo(\pi)=n-k}}F_{\DEX(\sigma),n}(\x)t^{\exc(\pi)+1}r^{n-k}.\\
    &=F_{\emptyset,n}(\x)r^n+\sum_{\pi\in\widetilde{\S}_n\setminus\{\theta\}}F_{\DEX(\sigma),n}(\x)t^{\exc(\pi)+1}r^{\fixTwo(\pi)}\\
    &=\sum_{\pi\in\widetilde{\S}_n}F_{\DEX(\sigma),n}(\x)t^{\exc(\pi)+1}r^{\fixTwo(\pi)}.  
\end{align*}
\end{proof}

Now we give a refinement of Theorem \ref{thm:TildeFY^jPermbasis} regarding $\widetilde{Q}_{n,j,k}(\x)$ that will be used to express $\widetilde{\Delta}_{n,r}(\x)$. 
Let the subset
\[ 
    \widetilde{FY}^{j,k}(\mathsf{B}_n)\coloneqq\left\{
    x_{F_1}^{a_1}\ldots x_{F_\ell}^{a_\ell}\in\widetilde{FY}^j(\mathsf{B}_n): |[n]-F_\ell|=k
    \right\}
\]
and $\widetilde{A}^{j,k}(\mathsf{B}_n)$ be the subspace spanned by $\widetilde{FY}^{j,k}(\mathsf{B}_n)$ which is obviously $\S_n$-invariant.

\begin{prop} \label{prop:RefinedZeroExtCode}
For any $0\le j\le n$ and $0\le k\le n$, we have 
\[
    \ch(\widetilde{A}^{j,k}(\mathsf{B}_n)_{\mathbb{C}})=\widetilde{Q}_{n,j,k}(\x).
\]
\end{prop}

\begin{proof}
To prove this result, we use the combinatorial objects called \emph{extended codes} in the binomial Eulerian story analogous to Stembridge codes in the Eulerian story. We only need extended codes in this proof. The readers are referred to Section 4.1.2 and Theorem 4.7 in \cite{Liao2023One} for the definition and an $\S_n$-equivariant bijection between extended codes and $\widetilde{FY}(\mathsf{B}_n)$.
The author \cite[Theorem 4.3]{Liao2023One} showed that the $\mathbb{C}\S_n$-module spanned by the set $\widetilde{\Code}_{n,j-1}$ of extended codes of index $j-1$ satisfying
\[
    \sum_{j=0}^{n}\ch\left(\mathbb{C}\widetilde{\Code}_{j-1}\right)t^j=\widetilde{Q}_n(\x,t).
\]
The same calculation actually shows that the set $\widetilde{\Code}_{n,j-1,k}$ of extended codes with index $j-1$ and $k$ $\infty$'s satisfies
\[
    \sum_{j=0}^n\sum_{k=0}^n\ch(\mathbb{C}\widetilde{\Code}_{n,j-1,k})t^j r^k=\widetilde{Q}_{n}(\x,t,r),
\]
or equivalently, $\ch(\mathbb{C}\widetilde{\Code}_{n,j-1,k})=\widetilde{Q}_{n,j,k}(\x)$. Now consider the $\S_n$-equivariant bijection $\widetilde{\phi}$ between $\widetilde{FY}(\mathsf{B}_n)$ and extended codes in \cite[Theorem 4.7]{Liao2023One}. The restriction of $\widetilde{\phi}$ to $\widetilde{FY}^{j,k}(\mathsf{B}_n)$ gives us an $\S_n$-equivariant bijection $\widetilde{FY}^{j,k}(\mathsf{B}_n)\longrightarrow\widetilde{\Code}_{n,j,k}$. Hence, we have
\[
    \ch(\widetilde{A}^{j,k}(\mathsf{B}_n))=\widetilde{Q}_{n,j,k}(\x).    
\]

\end{proof}

The following property of $\widetilde{Q}_{n,j,k}(\x)$ is analogous to Lemma \ref{lem:ExtractHQ} of $Q_{n,j,k}$.

\begin{lem} \label{lem:ExtractHTildeQ}
For all $n,j,k$,
\[
    \widetilde{Q}_{n,j,k}(\x)=h_k(\x)\widetilde{Q}_{n-k,j,0}(\x).
\]
\end{lem}

\begin{proof}
By Theorem \ref{thm:binEulerFix}, we have 
\begin{align*}
    h_k(\x)\widetilde{Q}_{n-k,j,0}(\x)
    =&F_{\emptyset,k}(\x)\sum_{\substack{\sigma\in\widetilde{\S}_{n-k} \\ \exc(\sigma)=j-1 \\ \fixTwo(\sigma)=0}} F_{\DEX(\sigma),n-k}(\x)
    =\sum_{\substack{\sigma\in\S_{n-k} \\ \exc(\sigma)=j-1}}F_{\DES(0^{k}),k}(\x) F_{\DES(\overline{\sigma}),n-k}(\x),
\end{align*}
where, recalled from Definition \ref{def:DEXext}, the barred permutation $\overline{\sigma}$ is obtained from $\sigma$ by adding a bar on $\sigma(i)$ for each $i\in\EXC(\sigma)$. Then the right-hand side of the above equality can be rewritten as
\[
    \sum_{\substack{\sigma\in\S_{n-k} \\ \exc(\sigma)=j-1 }}F_{\DES(0^{k}),k}(\x) F_{\DES(\overline{\sigma}),n-k}(\x)
    =\sum_{\substack{\sigma\in\S_{n-k} \\ \exc(\sigma)=j-1}}\sum_{\pi'\in 0^k\shuffle \overline{\sigma}}F_{\DES(\pi'),n}(\x)
\]
where $0^k\shuffle\overline{\sigma}$ is the set of shuffles of $\overline{\sigma}$ and $k$ $0$'s.
For example, the set $0\shuffle \overline{3}12=\{0\overline{3}12,\overline{3}012,\overline{3}102,\overline{3}120\}$.

Now, for any fixed $\sigma\in\S_{n-k}$, there is a bijection between $0^k\shuffle\overline{\sigma}$ and the set of decorated permutations of length $n$ with $k$ $0$'s whose nonzero part equals $\sigma$ after standardization. 
Let $\pi'$ be a word in $0^k\shuffle\overline{\sigma}$, the corresponding decorated permutation is given by leaving the $0$'s in $\pi'$ the same. 
Let $S$ be the indices of the positions that $0$'s occur in $\pi'$. Replace the nonzero part of $\pi'$ by a permutation on $[n]\setminus S$ whose standardization is $\sigma$. This procedure changes a word $\pi'\in 0^k\shuffle\overline{\sigma}$ to a decorated permutations $\pi$ such that $\DES(\pi')=\DES(\overline{\pi})=\DEX(\pi)$ and $\fixTwo(\pi)=k$.
Using our previous example, the words in the set $0\shuffle \overline{3}12$ correspond to the decorated permutations $0\overline{4}23$, $\overline{4}013$, $\overline{4}102$, $\overline{3}120$ in order.

Applying the bijection, we have 
\[
    \sum_{\substack{\sigma\in\S_{n-k} \\ \exc(\sigma)=j-1}}\sum_{\pi'\in 0^k\shuffle \overline{\sigma}}F_{\DES(\pi'),n}(\x)= \sum_{\substack{\pi\in\widetilde{\S}_n\\ \exc(\pi)=j-1 \\ \fixTwo(\sigma)=k}} F_{\DEX(\pi),n}(\x)
    =\widetilde{Q}_{n,j,k}(\x). 
\]
\end{proof}

Using the lemmas and propositions above, we are able to interpret the difference $\widetilde{\Delta}_{n,r}(\x)$ between $\grFrob\left(\widetilde{A}(U_{r+1,n}),t\right)$ and $\grFrob\left(\widetilde{A}(U_{r,n}),t\right)$.

\begin{prop} \label{prop:DifferenceAug}
For $1\le r\le n-1$, the polynomial $\widetilde{\Delta}_{n,r}(\x)\in\Lambda_n[t]$ has two expressions:
\[
    \sum_{j=0}^r h_{n-j}Q_j(\x,t)t^{r-j+1}
\]
and
\[
    \sum_{\substack{\sigma\in\widetilde{\S}_n \\ \fixTwo(\sigma)\ge n-r }}F_{\DEX(\sigma),n}(\x)t^{r-\exc(\sigma)}.
\]
\end{prop}

\begin{proof}
Similar to the case of Chow rings, for $1\le k\le n$ we have 
\[
    \widetilde{FY}^k(U_{1,n})\subsetneq \widetilde{FY}^k(U_{2,n})\subsetneq \ldots\subsetneq \widetilde{FY}^k(U_{n,n}). 
\]
The symmetric function $\widetilde{a}_{n,r}^{(k)}(\x)$ is the Frobenius characteristic of the permutation module generated by $\widetilde{FY}^k(U_{r+1,n})\setminus\widetilde{FY}^k(U_{r,n})$, which consists of monomial
\[
    x_{[n]}^{r+1} \quad \text{ if }k=r+1
\]
and monomials
\[
    x_{F_1}^{a_1}\ldots x_{F_\ell}^{a_\ell}x_{[n]}^i\in \widetilde{FY}^k(U_{r+1,n}) \quad\text{ such that }\quad |F_\ell|=r-i
\]
for $i=0,1,\ldots,k-1$ if $1\le k\le r$. Hence, $\widetilde{a}_{n,r}^{(r+1)}(\x)=h_n$. For $1\le k\le r$, One can construct an $\S_n$-equivariant bijection from $\widetilde{FY}^k(U_{r+1,n})\setminus\widetilde{FY}^k(U_{r,n})$ to $\bigcup_{i=0}^{k-1}\widetilde{FY}^{k-i,n-r-i}(\mathsf{B}_n)$ defined by $x_{F_1}^{a_1}\ldots x_{F_\ell}^{a_\ell}x_{[n]}^i\mapsto x_{F_1}^{a_1}\ldots x_{F_\ell}^{a_\ell}$. By Proposition \ref{prop:RefinedZeroExtCode}, we have
\[
    \widetilde{a}_{n,r}^{(k)}=\sum_{i=0}^{k-1}\widetilde{Q}_{n,k-i,n-r+i}(\x)=\sum_{i=0}^{k-1}h_{n-r+i}\widetilde{Q}_{r-i,k-i,0}(\x) \quad\text{(by Lemma \ref{lem:ExtractHTildeQ})}.
\]
Then since for any $n$ and $0\le j\le n-1$,
\[
    \widetilde{Q}_{n,j,0}(\x)=\sum_{\substack{\sigma\in\widetilde{\S}_{n}\\ \exc(\sigma)=j-1\\ \fixTwo(\sigma)=0}}F_{\DEX(\sigma),n}(\x)=\sum_{\substack{\sigma\in\S_{n}\\ \exc(\sigma)=j-1}}F_{\DEX(\sigma),n}(\x)=Q_{n,j-1}(\x),
\]
we have
\begin{align}
    \widetilde{\Delta}_{n,r}(\x)
    &=\sum_{k=1}^{r+1}\widetilde{a}_{n,r}^{(k)}(\x)t^k=h_nt^{r+1}+\sum_{k=1}^{r}\sum_{i=0}^{k-1}h_{n-r+i}\widetilde{Q}_{r-i,k-i,0}t^k \label{eq:differenceAug}\\
    &=h_nt^{r+1}+\sum_{k=1}^r\sum_{i=0}^{k-1}h_{n-r+i}Q_{r-i,k-i-1}t^k \nonumber\\
    &=h_nt^{r+1}+\sum_{i=0}^{r-1}h_{n-r+i}t^{i+1}\left(\sum_{k=i+1}^rQ_{r-i,k-i-1}t^{k-i-1}\right) \nonumber \\
    &=h_nt^{r+1}+\sum_{i=0}^{r-1}h_{n-r+i}t^{i+1}Q_{r-i}(\x,t) \nonumber\\
    &=\sum_{i=0}^r h_{n-r+i}Q_{r-i}(\x,t)t^{i+1} \nonumber\\
    &=\sum_{j=0}^r h_{n-j}Q_j(\x,t)t^{r-j+1}.  \nonumber
\end{align}

On the other hand, by Theorem \ref{thm:binEulerFix} and Lemma \ref{lem:PalinQ},  $\widetilde{Q}_{r-i,k-i,0}=Q_{r-i,k-i-1}(\x)=Q_{r-i,r-k}(\x)=\widetilde{Q}_{r-i,r-k+1,0}$. Therefore, equation (\ref{eq:differenceAug}) can be expressed in a different way as follows
\begin{align*}
    \widetilde{\Delta}_{n,r}(\x)
    &=h_nt^{r+1}+\sum_{k=1}^r\sum_{i=0}^{k-1}h_{n-r+i}\widetilde{Q}_{r-i,r-k+1,0}t^k\\
    &=h_nt^{r+1}+\sum_{k=1}^r\sum_{i=0}^{k-1}\widetilde{Q}_{n,r-k+1,n-r+i}t^k \quad\text{(by Lemma \ref{lem:ExtractHTildeQ}})\\
    &=h_nt^{r+1}+\sum_{k=1}^r\sum_{\substack{\sigma\in\widetilde{\S}_n\\ \exc(\sigma)=r-k \\ \fixTwo(\sigma)\ge n-r}}F_{\DEX(\sigma),n}(\x)t^k\\
    &=F_{\DEX(\theta),n}(\x)t^{r-(-1)}+\sum_{\substack{\sigma\in\widetilde{\S}_n\setminus\{\theta\} \\ \fixTwo(\sigma)\ge n-r }}F_{\DEX(\sigma),n}(\x)t^{r-\exc(\sigma)}\\
    &=\sum_{\substack{\sigma\in\widetilde{\S}_n \\ \fixTwo(\sigma)\ge n-r }}F_{\DEX(\sigma),n}(\x)t^{r-\exc(\sigma)}.
\end{align*}

\end{proof}

We are ready to prove Theorem \ref{thm:FrobAugUniform}, which we restate here.

\begin{thmn}[\ref{thm:FrobAugUniform}] \label{thm:FrobAugUniform2}
The graded Frobenius series of the augmented Chow ring of arbitrary uniform matroids $\grFrob\left(\widetilde{A}(U_{r,n})_{\mathbb{C}},t\right)$ has two expressions:
\begin{itemize}
    \item[(i)] \(\displaystyle h_n+t\sum_{j=0}^{r-1}h_{n-j}Q_j(\x,t)(1+t+\ldots+t^{r-j-1})\),
    \item[(ii)] \(\displaystyle \widetilde{Q}_n(\x,t)-\sum_{j=r}^{n-1}\sum_{\substack{\sigma\in\widetilde{\S}_n\\ \fixTwo(\sigma)\ge n-j}}F_{\DEX(\sigma),n}t^{j-\exc(\sigma)}\).
\end{itemize}
\end{thmn}

\begin{proof}
For (i), applying the first expression of Proposition \ref{prop:DifferenceAug}, we have
\begin{align*}
    \grFrob(\widetilde{A}(U_{r,n})_{\mathbb{C}},t)
    &=\grFrob(\widetilde{A}(U_{1,n}),t)+\sum_{j=1}^{r-1}\widetilde{\Delta}_{n,j}(\x)\\
    &=h_n(1+t)+\sum_{j=1}^{r-1}\sum_{i=0}^j h_{n-i}Q_i(\x,t)t^{j-i+1}\\
    &=h_n(1+t)+\sum_{j=1}^{r-1}h_n t^{j+1}+\sum_{j=1}^{r-1}\sum_{i=1}^j h_{n-i}Q_i(\x,t)t^{j-i+1}\\
    &=h_n+t\sum_{j=0}^{r-1}h_n t^{j}+t\sum_{i=1}^{r-1}h_{n-i}Q_i(\x,t)\sum_{j=i}^{r-1}t^{j-i}\\
    &=h_n+t\sum_{i=0}^{r-1}h_{n-i}Q_i(\x,t)(1+t+\ldots+t^{r-i-1}).
\end{align*}
For (ii), applying the second expression from Proposition \ref{prop:DifferenceAug}, we have
\begin{align*}
    \grFrob(\widetilde{A}(U_{r,n})_{\mathbb{C}},t)
    &=\grFrob(\widetilde{A}(U_{n,n})_{\mathbb{C}},t)-\sum_{j=r}^{n-1}\widetilde{\Delta}_{n,j}(\x)\\
    &=\widetilde{Q}_n(\x,t)-\sum_{j=r}^{n-1}\sum_{\substack{\sigma\in\widetilde{\S}_n \\ \fixTwo(\sigma)\ge n-j}}F_{\DEX(\sigma),n}(\x)t^{j-\exc(\sigma)}.
\end{align*}
\end{proof}

Then Corollary \ref{cor:FrobAugSpecial}, which we restate below, follows directly from Theorem \ref{thm:FrobAugUniform}.

\begin{corn}[\ref{cor:FrobAugSpecial}]
For the special cases that $r=n$ and $r=n-1$, we have
\begin{itemize}
    \item[(i)] \(\displaystyle \grFrob(\widetilde{A}(U_{n,n})_{\mathbb{C}},t)=\widetilde{Q}_n(\x,t),\)
    \item[(ii)]\(\displaystyle \grFrob(\widetilde{A}(U_{n-1,n})_{\mathbb{C}},t)=Q_n(\x,t).\)
\end{itemize}
\end{corn}

\begin{proof}
When $r=n$, nothing is being subtracted on the right-hand side of Theorem \ref{thm:FrobAugUniform2} (ii); hence (i) follows.

When $r=n-1$, by Theorem \ref{thm:FrobAugUniform2} (ii), we have
\begin{equation}
    \grFrob(\widetilde{A}(U_{n-1,n})_{\mathbb{C}},t)=\widetilde{Q}_n(\x,t)-\sum_{\substack{\sigma\in\widetilde{\S}_n\\ \fixTwo(\sigma)\ge 1}}F_{\DEX(\sigma),n}t^{n-1-\exc(\sigma)}. \label{eq:AugSpecial}
\end{equation}
Recall from Corollary 3.6 in \cite{ShareshianWachs2020} that $\widetilde{Q}_n(\x,t)$ as a polynomial in $\Lambda_n[t]$ is palindromic; hence 
\[
    \widetilde{Q}_n(\x,t)=\sum_{\sigma\in\widetilde{\S}_n}F_{\DEX(\sigma),n}t^{\exc(\sigma)+1}=\sum_{\sigma\in\widetilde{\S}_n}F_{\DEX(\sigma),n}t^{n-1-\exc(\sigma)}.
\]
Then (\ref{eq:AugSpecial}) can be rewritten as
\begin{align*}
    \grFrob(\widetilde{A}(U_{n-1,n})_{\mathbb{C}},t)
    &=\sum_{\substack{\sigma\in\widetilde{\S}_n\\ \fixTwo(\sigma)=0}}F_{\DEX(\sigma),n}t^{n-1-\exc(\sigma)}\\
    &=\sum_{\sigma\in\S_n}F_{\DEX(\sigma),n}t^{n-1-\exc(\sigma)}\\
    &=\sum_{\sigma\in\S_n}F_{\DEX(\sigma),n}t^{\exc(\sigma)} \quad \text{(palindromic by Lemma \ref{lem:PalinQ})}\\
    &=Q_n(\x,t).
\end{align*}
   
\end{proof}

\section{Parallelity between $U_{r,n}$ and $U_{r,n}(q)$} \label{Sec:Parellel}

Our goal in this section is to derive the formulas of the Hilbert series for the $q$-uniform matroids in Section \ref{subsec:q-uniform}. While one could adapt the proof strategy of Section \ref{Sec:ProofFrobChow} and \ref{Sec:ProofFrobAug}, we instead present a more conceptual argument using  Steinberg's results on unipotent representations of the general linear group $GL_n(q)\coloneqq GL_n(\mathbb{F}_q)$ from \cite{Steinberg1951geometric}. This approach yields Theorem \ref{thm:psToDim}, which simultaneously unifies the Frobenius-series formulas for the (augmented) Chow rings of $U_{r,n}$ and the Hilbert-series formulas for those of $U_{r,n}(q)$.

Recall from the final paragraph of Section \ref{subsec:ChowRings} that, for every $r$, the Chow ring  $A(U_{r,n})$ and the augmented Chow ring $\widetilde{A}(U_{r,n})$ are $\S_n$-permutation modules with permutation bases $FY(U_{r,n})$ and $\widetilde{FY}(U_{r,n})$ respectively. 
So are $A(U_{r,n})_{\mathbb{C}}$ and $\widetilde{A}(U_{r,n})_{\mathbb{C}}$. 
Since the basis elements are monomials indexed by chains in the Boolean lattice $B_n$, recalling our notation from Section \ref{subsec:RankSelection}, both $A(U_{r,n})_{\mathbb{C}}$ and $\widetilde{A}(U_{r,n})_{\mathbb{C}}$ are isomorphic to the direct sum of $\S_n$-permutation modules $\alpha_{B_n}(S)$. If $S=\{s_1<s_2<\ldots<s_\ell\}\subseteq [n-1]$, then 
\begin{equation} \label{eq:PermRepBooleanChain}
    \alpha_{B_n}(S)\cong_{\S_n} \mathbb{C}[\S_n/\S_{\nu}]\cong_{\S_n} \mathbf{1}_{\S_{\nu}}\uparrow_{\S_{\nu}}^{\S_n}, 
\end{equation}
where $\nu$ is the composition $\nu(S)=(s_1,s_2-s_1,\ldots,s_\ell-s_{\ell-1},n-s_{\ell})$ and $\S_{\nu}$ is the Young subgroup for the composition $\nu$. A well-known fact in the representation theory of $\S_n$ states that the induced representation (\ref{eq:PermRepBooleanChain}) can be decomposed into irreducible $\S_n$-modules $\{S^{\lambda}\}_{\lambda\vdash n}$ as
\begin{equation} \label{eq:MDecomposionSn}
    \mathbf{1}_{\S_{\nu}}\uparrow_{\S_{\nu}}^{\S_n}\cong_{\S_n} \bigoplus_{\lambda\vdash n}K_{\lambda,\nu}S^{\lambda}
\end{equation}
where $K_{\lambda,\nu}$ is the well-known \emph{Kostka number} (see Sagan \cite[Propostion 5.3.6]{Sagan2013symmetric}).

Similarly, since the general linear group $GL_n(q)$ acts on the $q$-Boolean lattice $B_n(q)$, i.e., the subspace lattice of $\left(\mathbb{F}_q\right)^n$, in an obvious way, both $A(U_{r,n}(q))_{\mathbb{C}}$ and $\widetilde{A}(U_{r,n}(q))_{\mathbb{C}}$ are permutation representations of $GL_n(q)$ with permutation bases $FY(U_{r,n}(q))$ and $\widetilde{FY}(U_{r,n}(q))$, respectively. Since the basis elements are monomials indexed by chains in $B_n(q)$, both $A(U_{r,n}(q))_{\mathbb{C}}$ and $\widetilde{A}(U_{r,n}(q))_{\mathbb{C}}$ are isomorphic to the direct sum of $GL_n(q)$-permutation modules $\alpha_{B_n(q)}(S)$. If $S=\{s_1<s_2<\ldots<s_\ell\}\subseteq [n-1]$, then 
\[
    \alpha_{B_n(q)}(S)\cong_{GL_n(q)} \mathbb{C}[GL_n(q)/P_{\nu}]\cong_{GL_n(q)}  \mathbf{1}_{P_{\nu}}\uparrow_{P_{\nu}}^{GL_n(q)}, 
\]
where $\nu=(s_1,s_2-s_1,\ldots,s_\ell-s_{\ell-1},n-s_{\ell})$ and $P_{\nu}$ is the maximal parabolic subgroup of $GL_n(q)$ that stabilizes the flag
\[
    \langle e_1,\ldots,e_{s_1}\rangle\subsetneq \langle e_1,\ldots, e_{s_1},\ldots,e_{s_2}\rangle\subsetneq\ldots\subsetneq \langle e_1,\ldots,e_{s_1},\ldots,e_{s_2},\ldots,e_{s_\ell}\rangle,
\]
i.e., the subgroup of matrices having main diagonal block matrices of size $s_1$,$s_2-s_1$,$\ldots$,$n-s_{\ell}$ and $0$ for entries below the main diagonal block matrices.

There is a parallelity between the representations of $\S_n$ and some representations of $GL_n(q)$. In the following, we will briefly describe this story, which mostly follows \cite[Section 5]{Stanley1982GroupPoset} and the lecture notes \cite[Section 4.2, 4.6, 4.7]{GrinbergReinerHopf2014}.

Steinberg \cite{Steinberg1951geometric} showed that the decomposition of $\mathbf{1_{P_{\nu}}}\uparrow_{P_{\nu}}^{GL_n(q)}$ into irreducible representations of $GL_n(q)$ is parallel to the decomposition of $\mathbf{1_{\S_{\nu}}}\uparrow_{\S_{\nu}}^{\S_n}$ given in (\ref{eq:MDecomposionSn}). 
More precisely, there is a distinguished family $\{S^{\lambda}_q\}$ of irreducible representations of $GL_n(q)$, indexed by partitions $\lambda$ of $n$, such that for any composition $\nu$ one has
\begin{equation} \label{eq:parallel}
    \mathbf{1}\uparrow_{P_{\nu}}^{GL_n(q)}\cong_{GL_n(q)} \bigoplus_{\lambda\vdash n}K_{\lambda,\nu}S^{\lambda}_q.    
\end{equation}
These $\{S^{\lambda}_q\}_{\lambda\vdash n}$, originally studied by Steinberg, are now referred to as the \emph{unipotent representations} of $GL_n(q)$. 
Let $\chi_q^\lambda$ denote the character of the unipotent representation $S_q^\lambda$ for each partition $\lambda$. For each $n$, let $\C(GL_n(q))$ be the free $\mathbb{Z}$-module spanned by the characters of unipotent reprentations of $GL_n(q)$ and define $\C(GL(q))\coloneqq \bigoplus_{n=0}^\infty \C(GL_n(q))$. 
One can make $\C(GL(q))$ a $\mathbb{Z}$-graded algebra via the \emph{parabolic induction product}:
\begin{align*}
    \C(GL_{n_1}(q))\times \C(GL_{n_2}(q)) &\longrightarrow \C(GL_{n_1+n_2}(q))\\ 
        (f_1, f_2) & \longmapsto f_1*f_2\coloneqq \left(f_1\otimes f_2 \Uparrow_{GL_{n_1}(q)\times GL_{n_2}(q)}^{P_{(n_1,n_2)}}\right)\uparrow_{P_{(n_1,n_2)}}^{GL_{n_1+n_2}(q)},
\end{align*}
where $(-)\Uparrow_{GL_{n_1}(q)\times GL_{n_2}(q)}^{P_{n_1,n_2}}$ is the \emph{inflation operation} that precomposes a representation of $GL_{n_1}(q)\times GL_{n_2}(q)$ with the surjective map $P_{n_1,n_2}\twoheadrightarrow GL_{n_1}(q)\times GL_{n_2}(q)$ given by 
\[
\left[\begin{array}{cc}
   A  &  C \\
   0  &  B
\end{array}\right]\mapsto 
\left[\begin{array}{cc}
   A  &  0 \\
   0  &  B
\end{array}\right].
\]
Because of the parallelity in (\ref{eq:parallel}), there is a $q$-analog of the Frobenius characteristic map, which is a $\mathbb{Z}$-graded isomorphism between $\C(GL(q))$ and the ring of symmetric functions $\Lambda_\mathbb{Z}$ over $\mathbb{Z}$,
\[
    \ch_q: \C(GL(q))\longrightarrow \Lambda_{\mathbb{Z}}
\]
sending $\ch_q(\chi^{\lambda}_q)=s_{\lambda}$ and $\ch_q(\mathbf{1}_{GL_n(q)})=h_n$. If $\chi$ is a character of a representation $V$, we sometimes write $\ch_q(V)$ to mean $\ch_q(\chi)$.

Hence for a graded $\mathbb{C}[GL_n(q)]$-module $W=\bigoplus_{j} W^j$ constituted by unipotent representations of $GL_n(q)$, one can define the \emph{$q$-graded Frobenius series} of $W$ by
\[
    \grFrob_q(W,t)\coloneqq\sum_{j}\ch_q(W^j)t^j.
\]

For any subset $S\subseteq [n-1]$ determining the composition $\nu=\nu(S)=(\nu_1,\ldots,\nu_{\ell+1})$, the induced representations satisfy
\[
    \ch_q\left(\alpha_{B_n(q)}(S)\right)=\ch_q\left(\mathbf{1}\uparrow_{P_{\nu}}^{GL_n(q)}\right)=h_{\nu_1}h_{\nu_2}\ldots h_{\nu_{\ell+1}}=\ch(\mathbf{1}_{\S_{\nu}}\uparrow_{\S_{\nu}}^{\S_n})=\ch\left(\alpha_{B_n}(S)\right).
\]
Moreover, the dimension of the induced representation of $GL_n(q)$ is given by the stable principal specialization of the corresponding induced representation of $\S_n$:
\begin{align} 
    \dim_{\mathbb{C}}(\mathbf{1}\uparrow_{P_{\nu}}^{GL_n(q)})
    =& |GL_n(q)/P_{\nu}|=\qbin{n}{\nu_1,\ldots,\nu_{\ell+1}}
    = \prod_{i=1}^n(1-q^i)
    \ps(h_{\nu}) \nonumber\\
    =& \prod_{i=1}^n(1-q^i)\ps\left(\ch(\mathbf{1}_{\S_{\nu}}\uparrow_{\S_{\nu}}^{\S_n})\right). \label{eq:DimIsPSQ}
\end{align}

Now let us revisit the structure of the representations on the Chow rings. Recall from (\ref{eq:FYChowBasis}) the combinatorial constraints on the monomials in $FY(U_{r,n})$ and $FY(U_{r,n}(q))$. Fix an integer $0\le i\le r-1$ and a subset $S = \{s_1 < s_2 < \cdots < s_\ell\}\subseteq [n-1]$. 
Then the multiplicities of $\alpha_{B_n}(S)$ in $A^i(U_{r,n})_{\mathbb{C}}$ and $\alpha_{B_n(q)}(S)$ in $A^i(U_{r,n}(q))_{\mathbb{C}}$ 
are both equal to the number of integer solutions \((a_1,\dots,a_\ell,b)\) of
\[  
    \left\{\begin{array}{l}
    a_1+a_2+\ldots+a_\ell+b=i\\
    1\le a_j\le s_j-s_{j-1}-1 \text{ for }j=1,\ldots,\ell \quad (\text{set }s_0\coloneqq 0)\\
    0\le b\le r-s_{\ell}-1
    \end{array}
    \right..
\] 
It follows that, under the usual Frobenius and $q$-Frobenius characteristic maps, the $\mathfrak{S}_n$-representation $A^i(U_{r,n})_{\mathbb{C}}$ and the $GL_n(q)$-representation $A^i(U_{r,n}(q))_{\mathbb{C}}$ have identical characters. The same holds for the representations on the augmented Chow rings.

\begin{prop} \label{prop:FrobIsQFrob}
For all $1\le r\le n$, we have
\[
    \grFrob_q(A(U_{r,n}(q))_{\mathbb{C}},t)=\grFrob(A(U_{r,n})_{\mathbb{C}},t), 
\]
\[
    \grFrob_q(\widetilde{A}(U_{r,n}(q))_{\mathbb{C}},t)=\grFrob(\widetilde{A}(U_{r,n})_{\mathbb{C}},t).
\]
\end{prop}

Moreover, combining the multiplicity argument with (\ref{eq:DimIsPSQ}) shows that, for every $i$, the $\mathbb{Q}$-dimensions of $A^i(U_{r,n}(q))$ and $\widetilde{A}^i(U_{r,n}(q))$ can be obtained by taking the stable principal specialization of $\ch(A^i(U_{r,n})_{\mathbb{C}})$ and $\ch(\widetilde{A}^i(U_{r,n})_{\mathbb{C}})$, respectively. 
Consequently, we arrive at the following result.
\begin{thm} \label{thm:psToDim}
For all $1\le r\le n$, the Hilbert series of the Chow ring and augmented Chow ring of arbitrary $q$-uniform matroids are given by
\[
    \Hilb\left(A(U_{r,n}(q)),t\right)=\prod_{i=1}^n(1-q^i)\ps\left(\grFrob\left(A(U_{r,n})_{\mathbb{C}},t\right)\right),
\]
\[
    \Hilb\left(\widetilde{A}(U_{r,n}(q)),t\right)=\prod_{i=1}^n(1-q^i)\ps\left(\grFrob\left(\widetilde{A}(U_{r,n})_{\mathbb{C}},t\right)\right).
\]
\end{thm}

By Theorem \ref{thm:psToDim} and the relations in Table \ref{table:relations}, applying the stable principle specialization to the graded Frobenius series for uniform matroids in Theorems \ref{thm:FrobChowUniform}, \ref{thm:FrobAugUniform} and Corollaries \ref{cor:FrobChowSpecial}, \ref{cor:FrobAugSpecial}, we derive the Hilbert series formula for $q$-uniform matroids in Theorems \ref{HRS:ChowQUniform}, \ref{thm:HilbQUniform}, \ref{thm:HilbAugQUniform} and Corollaries \ref{HRS:qHilbChowSpecial}, \ref{cor:SpeicalCaseHilbAug}.

Similarly, the stable principal specialization of the differences $\Delta_{n,r}(\x)$ and $\widetilde{\Delta}_{n,r}(\x)$ in Proposition \ref{prop:DifferenceChow} and \ref{prop:DifferenceAug} gives us the difference of the Hilbert series for $q$-uniform matroids. The results might be of interest, so we record them here.
 
\begin{cor} \label{cor:DiffQUniform}
For $1\le r\le n-1$, 
\begin{enumerate}
    \item[(i)] the difference $\Hilb(A(U_{r+1,n}(q)),t)-\Hilb(A(U_{r,n}(q)),t)$ has two expressions:
    \[
        \sum_{i=0}^r\qbin{n}{i}d_{i}(q,t)t^{r-i}\quad \text{ and }\quad \sum_{\substack{\sigma\in\S_n \\ \fix(\sigma)\ge n-r}}q^{\maj(\sigma)-\exc(\sigma)}t^{r-\exc(\sigma)}.
    \]
    \item[(ii)] the difference $\Hilb(\widetilde{A}(U_{r+1,n}(q)),t)-\Hilb(\widetilde{A}(U_{r,n}(q)),t)$ has two expressions:
    \[
        \sum_{i=0}^r\qbin{n}{i}A_i(q,t)t^{r+1-i} \quad \text{ and }\quad \sum_{\substack{\sigma\in\widetilde{\S}_n\\ \fixTwo(\sigma)\ge n-r}}q^{\maj(\sigma)-\exc(\sigma)}t^{r-\exc(\sigma)}.
    \]    
\end{enumerate}

\end{cor}



\section{Charney--Davis quantities for (augmented) Chow rings of matroids} \label{Sec:GeneralEQ}

We prove our two general theorems, Theorem \ref{thm:EqCDChow} and Theorem \ref{thm:EqCDAug}, of the equivariant Charney--Davis quantities for general matroids in Section \ref{subsec:generalCD}. Then in Section \ref{subsec:CDUniform} we consider the special cases of uniform matroids and $q$-uniform matroids.

\subsection{Charney--Davis quantities for general matroids} \label{subsec:generalCD}

Let $M$ be a matroid on $E=[n]$ of rank $r$ with the lattice of flats $\L(M)$. Let $G$ be any subgroup of $\Aut(M)$. In this section, we will prove Theorem \ref{thm:EqCDChow} and Theorem \ref{thm:EqCDAug}, which are our formulas for the $G$-equivariant Charney--Davis quantities of (augmented) Chow rings of general matroids.

Recall from Section \ref{subsec:RankSelection} the background on rank-selected subposets.
Consider the permutation representation $\alpha_{\L(M)}(S)$ and the virtual representation $\beta_{\L(M)}(S)$ of $G$ for $S\subseteq [r-1]$. We have, by Theorem \ref{thm:BetaGenRep}, that 
\begin{equation}  \label{eq:IsHomology}
    \beta_{\L(M)}(S)\cong_{G}\tilde{H}_{|S|-1}(\L(M)_S).
\end{equation}

On the other hand, the action of $G$ on $\L(M)$ also induces actions on the Chow ring $A(M)$ and the augmented Chow ring $\widetilde{A}(M)$. 
Recall from (\ref{thm:FYChowBasis}) and Corollary \ref{thm:FYAugBasis} that the Feichtner--Yuzvinsky bases for the $A(M)$ and $\widetilde{A}(M)$, respectively, are
\[
    FY(M)=\left\{x_{F_1}^{a_1}x_{F_2}^{a_2}\ldots x_{F_\ell}^{a_\ell}:\substack{~\emptyset=F_0\subsetneq F_1\subsetneq F_2\subsetneq\ldots\subsetneq F_\ell\subseteq \L(M) \text{ for } 0\le \ell\le n \\
    1\le a_i\le \rk_M(F_i)-\rk_M(F_{i-1})-1}\right\}
\]
and 
\[
	\widetilde{FY}(M)\coloneqq\left\{x_{F_1}^{a_1}x_{F_2}^{a_2}\ldots x_{F_\ell}^{a_\ell}: \substack{\emptyset\subsetneq F_1\subsetneq F_2\subsetneq\ldots\subsetneq F_\ell \text{ for }0\le \ell\le n \\
    1\le a_1\le\rk_M(F_1),~ a_i\le\rk_M(F_i)-\rk_M(F_{i-1})-1 \text{ for }i\ge 2}
    \right\}.
\]
Similar to the case of uniform matroids described in Section \ref{Sec:Parellel}, the subsets $FY^k(M)$ and $\widetilde{FY}^k(M)$ of degree $k$ monomials form permutation bases for the $G$-module $A^k(M)$ and $\widetilde{A}^k(M)$ respectively for every $k$.
Therefore, as representations of $G$, both $A(M)$ and $\widetilde{A}(M)$ are the direct sums of permutation representations of the form $\alpha_{\L(M)}(S)$ for some $S\subseteq [r-1]$.

Now we are ready to prove Theorem \ref{thm:EqCDChow}, which we restate here.

\begin{thmn}[\ref{thm:EqCDChow}] 
Let $M$ be a (loopless) matroid on $[n]$ of rank $r$, 
\[
    \sum_{i=0}^{r-1}(-1)^i A^i(M)_{\mathbb{C}}\cong_{G}
    \begin{cases}
        0   &   \text{ if $r$ is even}\\
        (-1)^{\frac{r-1}{2}}\tilde{H}_{\frac{r-3}{2}}\left(\L(M)_{\Even(r-1)}\right)
         & \text{ if $r$ is odd}
    \end{cases}.
\]
\end{thmn}
\begin{proof}
Following from the constraints of exponents of the FY-basis elements in $FY(M)$, we have
\begin{equation} \label{eq:RepA(M)}
    \sum_{i=0}^{r-1}A^i(M)_{\mathbb{C}}t^i
    \cong_G \sum_{S=\{r_1<\ldots<r_\ell\}\subseteq [r]}\prod_{j=1}^\ell \frac{t(1-t^{r_i-r_{i-1}-1})}{1-t}\alpha_{\L(M)}(S)
\end{equation}
with the convention that $r_0\coloneqq 0$. 
Notice that here we allow $S$ to be any subset of $[r]$, because the definiton of $\alpha_{\L(M)}(S)$ can be easily extended to $S$ containing $r$ and $\alpha_{\L(M)}(S)\cong_G\alpha_{\L(M)}(S\setminus\{r\})$ for any $S$ containing $r$. Then (\ref{eq:RepA(M)}) can be rewritten as
\begin{align*}
    &\sum_{i=0}^{r-1}A^i(M)_{\mathbb{C}}t^i\\
    \cong_G  &\sum_{\emptyset\neq S=\{r_1<\ldots<r_\ell\}\subseteq [r-1]}\prod_{j=1}^{\ell}\frac{t(1-t^{r_j-r_{j-1}-1})}{1-t}\left(1+\frac{t(1-t^{r-r_{\ell}-1})}{1-t}\right)\alpha_{\L(M)}(S)\\
    & +\left(1+\frac{t(1-t^{r-1})}{1-t}\right)\alpha_{\L(M)}(\emptyset).
\end{align*}     
Evaluating the above identity in $R_{\mathbb{C}}(G)[t]$ at $t=-1$, for a term in the summation to be nonzero, all $r_i$'s are forced to be even. The product $\prod_{j=1}^{\ell}\frac{t(1-t^{r_j-r_{j-1}-1})}{1-t}$ evaluated at $t=-1$ gives $(-1)^{\ell}$.  If $r$ is also even, then both $1+\frac{t(1-t^{r-1})}{1-t}$ and $1+\frac{t(1-t^{r-r_{\ell}-1})}{1-t}$ evaluated at $t=-1$ give
\[
    1+\frac{(-1)(1-(-1))}{1-(-1)}=1+(-1)=0.
\]
Therefore, $\sum_{i=0}^{r-1}(-1)^i A^i(M)_{\mathbb{C}}=0$ if $r$ is even.

If $r$ is odd, then both $\frac{t(1-t^{r-1})}{1-t}$ and $\frac{t(1-t^{r-r_{\ell}-1})}{1-t}$ vanish after evaluated at $t=-1$. In this case, as a virtual representation of $G$, the evaluation at $t=-1$ gives
\begin{align*}
    \sum_{i=0}^{r-1}(-1)^i A^i(M)_{\mathbb{C}} 
    &\cong_G \sum_{S\subseteq\Even(r-1)}(-1)^{|S|}\alpha_{\L(M)}(S)\\
    &=(-1)^{\frac{r-1}{2}}\sum_{S\subseteq\Even(r-1)}(-1)^{\frac{r-1}{2}-|S|}\alpha_{\L(M)}(S)\\
    &=(-1)^{\frac{r-1}{2}}\beta_{\L(M)}(\Even(r-1))\\
    &\cong_G (-1)^{\frac{r-1}{2}}\tilde{H}_{\frac{r-3}{2}}\left(\L(M)_{\Even(r-1)}\right) \quad (\text{by }(\ref{eq:IsHomology})).
\end{align*}
\end{proof}

Next, we prove the augmented analog of Theorem \ref{thm:EqCDChow}, Theorem \ref{thm:EqCDAug}, which we restate below. 
Recall that we have a similar $G$-action on the FY-basis $\widetilde{FY}(M)$ for $\widetilde{A}(M)$. 
\begin{thmn}[\ref{thm:EqCDAug}]
Let $M$ be a matroid on $[n]$ of rank $r$. Then
\[
    \sum_{i=0}^{r}(-1)^i\widetilde{A}^i(M)_{\mathbb{C}}\cong_{G}
    \begin{cases}
        (-1)^{\frac{r}{2}}\tilde{H}_{\frac{r-2}{2}}\left(\L(M)_{\Odd(r-1)}\right)   &   \text{ if $r$ is even}\\
        0   & \text{ if $r$ is odd}
    \end{cases}.
\]
\end{thmn}

\begin{proof}
The proof is similar to that of Theorem \ref{thm:EqCDChow}. Following from the restriction of the exponents of the FY-basis elements in $\widetilde{FY}(M)$, we have
\begin{align*}
    &\sum_{i=0}^r\widetilde{A}^i(M)_{\mathbb{C}}t^i
    \cong_G \sum_{S=\{r_1<\ldots<r_\ell\}\subseteq [r]}\frac{t(1-t^{r_1})}{1-t}\prod_{j=2}^\ell \frac{t(1-t^{r_i-r_{i-1}-1})}{1-t}\alpha_{\L(M)}(S)\\
    =& \sum_{\emptyset\neq S=\{r_1<\ldots<r_\ell\}\subseteq [r-1]}\frac{t(1-t^{r_1})}{1-t}\prod_{j=2}^{\ell}\frac{t(1-t^{r_j-r_{j-1}-1})}{1-t}\left(1+\frac{t(1-t^{r-r_{\ell}-1})}{1-t}\right)\alpha_{\L(M)}(S)\\
    & +\left(1+\frac{t(1-t^{r})}{1-t}\right)\alpha_{\L(M)}(\emptyset).
\end{align*}
After evaluating at $t=-1$, the right-hand side of the last identity is nonvanishing only when all $r_i$'s are odd and $r$ is even. Thus, in this case, the $(-1)$-evaluation gives
\begin{align*}
    \sum_{i=0}^{r}(-1)^i\widetilde{A}(M)_{\mathbb{C}}
    &\cong_G \sum_{S\subseteq\Odd(r-1)}(-1)^{|S|}\alpha_{\L(M)}(S)\\
    &=(-1)^{\frac{r}{2}}\sum_{S\subseteq\Odd(r-1)}(-1)^{\frac{r}{2}-|S|}\alpha_{\L(M)}(S)\\
    &=(-1)^{\frac{r}{2}}\beta_{\L(M)}(\Odd(r-1))\\
    &\cong_G (-1)^{\frac{r}{2}}\tilde{H}_{\frac{r-2}{2}}(\L(M)_{\Odd(r-1)}) \quad (\text{by }(\ref{eq:IsHomology}))
\end{align*}

\end{proof}

\subsection{Special cases of uniform matroids and their $q$-analogs} \label{subsec:CDUniform}

Let $M$ be the uniform matroid $U_{r,n}$. Then the automorphism group $\Aut(U_{r,n})$ is the symmetric group $\S_n$. Note that $\L(U_{r,n})$ is the rank-selected poset ${(B_n)}_{[r-1]}$ of the Boolean lattice $B_n$. Then by Theorem \ref{thm:RankHomologyBn},
\[
    \tilde{H}_{\frac{r-3}{2}}\left(\L(U_{r,n})_{\Even(r-1)}\right)=\tilde{H}_{\frac{r-3}{2}}\left({(B_n)}_{\Even(r-1)}\right)\cong_{\S_n} S^{H_{\Even(r-1),n}}.
\]
Theorem \ref{thm:EqCDChow} in the case of uniform matroids $U_{r,n}$ together with the above isomorphism yields the following new results. 

\begin{thm} \label{thm:CDEqChowUniform}
For any positive integer $n$ and $1\le r\le n$, the $\S_n$-equivariant Charney--Davis quantity of the Chow ring of a uniform matroid is given by
\[
    \sum_{i=0}^{r-1}(-1)^i A^i(U_{r,n})_{\mathbb{C}}\cong_{\S_n}
    \begin{cases}
        0   &   \text{ if $r$ is even}\\
        (-1)^{\frac{r-1}{2}}S^{H_{\Even(r-1),n}}
         & \text{ if $r$ is odd}
    \end{cases}.
\]
or, equivalently, as graded Frobenius series
\[
    \grFrob(A(U_{r,n})_{\mathbb{C}},-1)=
    \begin{cases}
        0   &   \text{ if $r$ is even}\\
        (-1)^{\frac{r-1}{2}}s_{H_{\Even(r-1),n}}
         & \text{ if $r$ is odd}
    \end{cases}.
\]
where $s_{D}$ is the Schur function of skew shape $D$.
\end{thm}

Similarly, for the augmented case, we obtain the following. 

\begin{thm} \label{thm:CDEqAugUniform}
For any positive integer $n$ and $1\le r\le n$, the $\S_n$-equivariant Charney--Davis quantity of the augmented Chow ring of a uniform matroid is given by
\[
    \sum_{i=0}^{r}(-1)^i \widetilde{A}^i(U_{r,n})_{\mathbb{C}}\cong_{\S_n}
    \begin{cases}
        (-1)^{\frac{r}{2}}S^{H_{\Odd(r-1),n}}   &   \text{ if $r$ is even}\\
        0    & \text{ if $r$ is odd}
    \end{cases},
\]
or, equivalently, as graded Frobenius series
\[
    \grFrob(\widetilde{A}(U_{r,n})_{\mathbb{C}},-1)=
    \begin{cases}
        (-1)^{\frac{r}{2}}s_{H_{\Odd(r-1),n}}   &   \text{ if $r$ is even}\\
        0,   & \text{ if $r$ is odd}
    \end{cases}
\]
\end{thm}

Apply $\prod_{i=1}^n(1-q^i)\ps(-)$ to the symmetric function identities in Theorem \ref{thm:CDEqChowUniform} and \ref{thm:CDEqAugUniform}. Since for $R\subseteq [n-1]$, we have
\[
    \prod_{i=1}^n(1-q^i)\ps(s_{R,n})=\sum_{\substack{\sigma\in\S_n \\ \DES(\sigma)=R}}q^{\maj(\sigma^{-1})}=\sum_{\substack{\sigma\in\S_n \\ \DES(\sigma)=R}}q^{\inv(\sigma)}
\]
(see \cite[Theorem 4.4]{ShareshianWachs2020} and its proof for a detailed discussion), then Theorem \ref{thm:psToDim} together with the above identity leads to our formulas of the Charney--Davis quantities for $q$-uniform matroids $U_{r,n}(q)$. 
In Theorem \ref{thm:SignedQUniform} (i) below, Hameister, Rao, and Simpson's formula (\ref{HRS:sign}) of the Charney--Davis quantity is simplified to an elegant combinatorial interpretation, which is new even in the case of $q=1$.

\begin{thm} \label{thm:SignedQUniform}
For any positive integer $n$ and $1\le r\le n$, we have:
\begin{enumerate}
    \item[(i)] The Charney--Davis quantity of the Chow ring of an arbitrary $q$-uniform matroid is given by
    \[
    \Hilb(A(U_{r,n}(q)),-1)=\begin{cases}
        0  & \mbox{ if $r$ is even}\\
        \displaystyle (-1)^{\frac{r-1}{2}}\!\!\!\!\!\sum_{\substack{\sigma\in\S_n\\ \DES(\sigma)=\Even(r-1)}}q^{\inv(\sigma)} & \mbox{ if $r$ is odd}
        \end{cases}.
    \]
    \item[(ii)] The Charney--Davis quantity of the augmented Chow ring of an arbitrary $q$-uniform matroid is given by
    \[
    \Hilb(\widetilde{A}(U_{r,n}(q)),-1)=\begin{cases}
         \displaystyle (-1)^{\frac{r}{2}}\sum_{\substack{\sigma\in\S_n\\ \DES(\sigma)=\Odd(r-1)}}q^{\inv(\sigma)} & \mbox{ if $r$ is even}\\
        0  & \mbox{if $r$ is odd}
    \end{cases}.
    \]
\end{enumerate}
In particular, both $CD(A(U_{r,n}(q)))$ and $CD(\widetilde{A}\left(U_{r,n}(q))\right)$ are polynomials in $\mathbb{Z}_{\ge 0}[q]$.
\end{thm}

\begin{rem}
Following from Stanley \cite[Example 2.2.4]{StanleyEC1}, the $q$-polynomial $\sum_{\substack{\sigma\in\S_n\\ \DES(\sigma)=\Even(r-1)}}q^{\inv(\sigma)}$ in Theorem \ref{thm:SignedQUniform} (i) can be expressed as one determinant
\[
    [n]_q!\left|
    \begin{array}{cccccc}
        \frac{1}{[2]_q!} &         1        &   0   & \ldots & 0  & 0\\
        \frac{1}{[4]_q!} & \frac{1}{[2]_q!} &   1   & \ldots & 0  & 0\\
        \frac{1}{[6]_q!} & \frac{1}{[4]_q!} &\frac{1}{[2]_q!} &\ldots & 0 & 0\\
        \vdots           & \vdots           & \vdots& \ddots & \vdots & \vdots\\
        \frac{1}{[r-1]_q!} & \frac{1}{[r-3]_q!} & \frac{1}{[r-5]_q!} & \ldots & \frac{1}{[2]_q!} & 1\\
        \frac{1}{[n]_q!} & \frac{1}{[n-2]_q!}& \frac{1}{[n-4]_q!} & \ldots & \frac{1}{[n-r+3]_q!} & \frac{1}{[n-r+1]_q!}
\end{array}
    \right|,
\]
which simplifies Hameister, Rao, and Simpson's formula (\ref{HRS:determinant}). The $q$-polynomial can also be obtained by realizing the alternating sum in (\ref{HRS:sign}) as a process of inclusion-exclusion.
\end{rem}

\begin{rem}
Combining with Corollary \ref{HRS:qHilbChowSpecial}, Theorem \ref{thm:SignedQUniform} (i) generalizes Foata and Han's result in \cite[Theorem 1]{FoataHan2010}. 
\end{rem}

\begin{rem}
Similarly, the $q$-polynomial $\sum_{\substack{\sigma\in\S_n\\ \DES(\sigma)=\Odd(r-1)}}q^{\inv(\sigma)}$ in Theorem \ref{thm:SignedQUniform} (ii) can be expressed as 
\[
    [n]_q!\left|
    \begin{array}{cccccc}
        \frac{1}{[1]_q!} &         1        &   0   & \ldots & 0  & 0\\
        \frac{1}{[3]_q!} & \frac{1}{[2]_q!} &   1   & \ldots & 0  & 0\\
        \frac{1}{[5]_q!} & \frac{1}{[4]_q!} &\frac{1}{[2]_q!} &\ldots & 0 & 0\\
        \vdots           & \vdots           & \vdots& \ddots & \vdots & \vdots\\
        \frac{1}{[r-1]_q!} & \frac{1}{[r-2]_q!} & \frac{1}{[r-4]_q!} & \ldots & \frac{1}{[2]_q!} & 1\\
        \frac{1}{[n]_q!} & \frac{1}{[n-1]_q!}& \frac{1}{[n-3]_q!} & \ldots & \frac{1}{[n-r+3]_q!} & \frac{1}{[n-r+1]_q!}
\end{array}
    \right|.
\]
On the other hand, one can obtain the following augmented analog of Hameister, Rao, and Simpson's formula (\ref{HRS:sign}) by the principle of inclusion-exclusion 
\[ 
    (-1)^{\frac{r}{2}}\sum_{\substack{\sigma\in\S_n\\ \DES(\sigma)=\Odd(r-1)}}q^{\inv(\sigma)}=1+\sum_{k=0}^{\frac{r-2}{2}}(-1)^{k+1}\qbin{n}{2k+1}E^*_{2k+1}(q),
\]
where $E^*_{2k+1}(q)=\sum_{\sigma\in\Alt_{2k+1}}q^{\inv(\sigma)}$ and $
\Alt_{2k+1}=\{\sigma\in\S_{2k+1}:\sigma_1>\sigma_2<\sigma_3\ldots<\sigma_{2k+1}\}$ is the set of \emph{alternating permutations} in $\S_{2k+1}$. Note that $E_n(q)$ and $E^*_n(q)$ are related by $E^*_{n}(q)=q^{\binom{n}{2}}E_n(q)$.

\end{rem}

Again, because Theorem \ref{thm:psToDim} tells us that the Hilbert series for $U_{r,n}(q)$ agrees with the stable principle specialization of the Frobenius characteristics for $U_{r,n}$, letting $q=1$ gives us the Hilbert series for $U_{r,n}$. 
Therefore, we obtain the Charney--Davis quantities for $U_{r,n}$ by setting $q=1$ in Theorem \ref{thm:SignedQUniform}.

\begin{cor} \label{cor:SignedUniform}
For any positive integer $n$ and $1\le r\le n$, we have:
\begin{enumerate}
    \item[(i)] The Charney--Davis quantity of the Chow ring of an arbitrary uniform matroid is given by
    \[
        \Hilb(A(U_{r,n}),-1)=
        \begin{cases}
        0  & \mbox{ if $r$ is even}\\
        (-1)^{\frac{r-1}{2}}|\{\sigma\in\S_n: \DES(\sigma)=\Even(r-1)\}| & \mbox{ if $r$ is odd}
        \end{cases}.
    \]
    \item[(ii)] The Charney--Davis quantity of the augmented Chow ring of an arbitrary uniform matroid is given by
    \[
        \Hilb\left(\widetilde{A}(U_{r,n}),-1\right)
        =\begin{cases}
        (-1)^{\frac{r}{2}}|\{\sigma\in\S_n: \DES(\sigma)=\Odd(r-1)\}|  & \mbox{ if $r$ is even}\\
        0 & \mbox{ if $r$ is odd}
    \end{cases}.   
    \]
\end{enumerate}
In particular, both $CD(A(U_{r,n}))$ and $CD\left(\widetilde{A}(U_{r,n})\right)$ are nonnegative. 
\end{cor}

\begin{rem}
When $r=n$, the Hilbert series of $A(U_{n,n})$ and $\widetilde{A}(U_{n,n})$ are exactly the $h$-polynomials of the permutahedron $\mathsf{Perm}_n$ and the stellohedron $\mathsf{St}_n$, respectively. Hence Corollary \ref{cor:SignedUniform} (i) and (ii) recover the facts that the Charney--Davis quantities of the permutahedron and the stellohedron give the tangent number and the secant number, respectively:
\[
    CD(\mathsf{Perm}_n)=\begin{cases}
    0 & \text{ if $n$ is even}\\
    E_n & \text{ if $n$ is odd}
    \end{cases}; \qquad
    CD(\mathsf{St}_n)=\begin{cases}
    E_n & \text{ if $n$ is even}\\
    0   & \text{ if $n$ is odd}
    \end{cases}.
\]

\end{rem}

\section{Some Further developements} \label{sec:Final}

Within the framework of palindromic polynomials, there is a generalization of the nonnegativity of the Charney--Davis quantity called \emph{$\gamma$-positivity} or \emph{$\gamma$-nonnegativity}.  
It has been shown independently by Ferroni, Matherne, Stevens, Vecchi \cite[Theorem 3.25]{Ferroni2022hilbert} and Wang (see \cite[p. 29]{Ferroni2022hilbert}) that the Hilbert series of $A(M)$ and $\widetilde{A}(M)$ are $\gamma$-positive, which follows from the work of Braden, Huh, Matherne, Proudfoot, and Wang \cite{BHMPW2020} on the \emph{semismall decomposition} of Chow rings. 

One also has a notion of \emph{equivariant $\gamma$-positivity} for a graded representation $A=\bigoplus_{i=0}^r A^i$, which was first considered by Shareshian and Wachs \cite[Section 5]{ShareshianWachs2020}.
Consider the subgroup $G$ of $\Aut(M)$ acting on $A(M)$ and $\widetilde{A}(M)$. It has been conjectured by Angarone, Nathanson, and Reiner \cite[Conjecture 5.2]{ANR2023PermRepchow} that $A(M)_{\mathbb{C}}=\bigoplus_{i=0}^{r-1}A^i(M)$ is $G$-equivariant $\gamma$-positive. In \cite{Liao2024EqGamma}, Liao generalized Theorem \ref{thm:EqCDChow} and \ref{thm:EqCDAug}, showing that both $A(M)$ and $\widetilde{A}(M)$ are $G$-equivariant $\gamma$-positive. This result confirms the conjecture of Angarone, Nathanson, and Reiner.

Since the preliminary version of this paper appeared on $\texttt{arXiv}$, there have been several interesting developments beyond the equivariant $\gamma$-positivity results previously mentioned.

Ferroni, Matherne, and Vecchi \cite{FMVChowPoly} generalized the notion of the (augmented) Chow polynomial of a matroid—that is, the Hilbert series of the (augmented) Chow ring—to any finite graded poset and integrated it into the Kazhdan--Lusztig--Stanley theory (see \cite{localhStanley1992, ProudfootKLSpoly2018}). Theorems \ref{thm:EqCDChow} and \ref{thm:EqCDAug} can be extended to the equivariant Chow polynomial for Cohen--Macaulay posets using the same proof. 

Hoster \cite{Hoster2024chow} gave interpretations to the (augmented) Chow polynomials of $U_{r,n}$, distinct from Theorems \ref{HRS:ChowQUniform} and \ref{thm:HilbAugQUniform}, in terms of \emph{``permuted'' Schubert matroids} with \emph{cogirths} greater than $n-r$. 
Since cogirth is invariant under matroid isomorphism, this interpretation can be extended to the $h$-expansion of $\S_n$-equivariant (augmented) Chow polynomials.
In the case of $U_{n,n}$, the interpretations can be proven via an $\S_n$-equivariant bijection from $\widetilde{FY}(U_{n,n})$ to permuted Schubert matroids, as given by Eur, Huh, and Larson \cite[Proof of Theorem 7.7]{EHL2022stellahedral}. However, restricting this bijection to $\widetilde{FY}(U_{r,n})$ does not yield Hoster's interpretation. It would be interesting to characterize the image of this restriction and compare it with Hoster's interpretation.

\section*{Acknowledgements}
This work is part of the author's PhD thesis at the University of Miami. 
The author thanks Michelle Wachs for her guidance and encouragement at every stage of this project, and Patricia Commins and Richard Stanley for helpful conversations on unipotent representations of $GL_n(q)$. The author thanks the anonymous referees for carefully reviewing the manuscript and providing suggestions that significantly improved the exposition.

\bibliographystyle{alpha}
\bibliography{bibliography}

\end{document}